\documentclass[11pt]{article} 

\usepackage[utf8]{inputenc} 

\usepackage{geometry} 
\usepackage{graphicx} 

\usepackage{booktabs} 
\usepackage{array} 
\usepackage{paralist} 
\usepackage{verbatim} 
\usepackage{subfig} 

\usepackage{fancyhdr} 
\usepackage{sectsty}
\allsectionsfont{\sffamily\mdseries\upshape} 

\usepackage[nottoc,notlof,notlot]{tocbibind} 
\usepackage[titles,subfigure]{tocloft} 

\usepackage{amsmath, amssymb, bm}
\usepackage[T1]{fontenc}
\usepackage{textcomp}
\usepackage{mathrsfs}

\newtheorem{theorem}{Theorem}[section]
\newtheorem{lemma}[theorem]{Lemma}
\newtheorem{proposition}[theorem]{Proposition}

\numberwithin{equation}{section} 

\title{Jucys--Murphy Elements for Wreath Products and \\
Their Application to Dynamical Random Multi-Diagrams
\thanks{
This work was supported by JSPS KAKENHI Grant Number JP22K03346} 
}
\author{{Akihito HORA}
\thanks{Department of Mathematics, Faculty of Science, Hokkaido University, Sapporo 
060-0810, Japan; hora@math.sci.hokudai.ac.jp}}

\begin{document}
\maketitle

\begin{abstract}
The equivalence classes of irreducible representations of wreath product 
$\mathfrak{S}_n(T) = T^n \rtimes \mathfrak{S}_n$ of finite group $T$ with respect 
to symmetric group $\mathfrak{S}_n$ are parametrized by 
$\mathbb{Y}_n(\widehat{T})$, the $\lvert \widehat{T}\rvert$-tuple Young diagrams 
with total size $n$. 
We show a formula connecting the Kerov transition measures of these Young diagrams 
with the Jucys--Murphy elements of $\mathfrak{S}_n(T)$. 
This formula is due to Biane in the case of symmetric groups. 
The formula enables us to investigate asymptotic property of the shapes of 
multi-diagrams through combinatorial analysis for the Jucys--Murphy elements. 
On the other hand, a Markov chain is introduced on $\mathbb{Y}_n(\widehat{T})$, 
canonically reflecting the branching rule for the tower of wreath product groups. 
We have a continuous time stochastic process on $\mathbb{Y}_n(\widehat{T})$ 
from this chain by replacing the discrete time by a counting process. 
Our project is to specify the deterministic limit shape of multi-diagrams at each 
macroscopic time through appropriate space-time scaling limit, and to describe 
evolution of related quantities characterizing the shape. 
Especially, we derive dynamical concentrated limit shapes in the case of abelian $T$ 
by using free probability tools under the assumption of approximate factorization property 
for initial ensembles with an additional property of a pausing time distribution.
\end{abstract}

\section{Introduction}
Jucys--Murphy elements (\textit{JM elements}, for short) of the symmetric group $\mathfrak{S}_n$ 
\begin{equation}\label{eq:1-1}
J_1 := 0, \qquad J_k := (1\ k)+(2\ k) + \cdots + (k\!-\!1\ k) \in \mathbb{C}[\mathfrak{S}_n] 
\quad (2\leqq k\leqq n) 
\end{equation}
appear in a wide variety of contexts related to representations of symmetric groups, 
including fundamental textbooks. 
A series of seminal works due to Biane reveal that JM elements play an 
important role in interplay between asymptotic representation theory of symmetric groups and 
free probability theory (\cite{Bia95}, \cite{Bia98}, \cite{Bia01}, \cite{Bia03}). 
In relation to the limit shapes of Young diagrams (\textit{diagrams}, for short), 
Biane's formula connecting JM elements with Kerov's transition measure is crucial. 
Let us quickly recall the formula. 
See Table~\ref{tab:1-1} for some notations on diagrams used in the present paper. 

\begin{table}[htb]
\begin{tabular}{ll} \toprule 
the set of diagrams & $\mathbb{Y}$ \\ 
the set of diagrams of size $n$ & $\mathbb{Y}_n$ \\ 
size (weight) of diagram $\nu$ & $|\nu|$ \\ 
number of parts (rows) of diagram $\nu$ & $l(\nu)$ \\ 
number of parts of length $j$ of diagram $\nu$ & $m_j(\nu)$ \\ 
adding or removing a box of diagrams & 
$\nu\nearrow\mu$ or $\mu\searrow\nu$ \\ \bottomrule 
\end{tabular}
\caption{Notations for Young diagrams. e.g. 
$(1^12^23^1)\in\mathbb{Y}_8 \nearrow (1^12^13^2)\in\mathbb{Y}_9$}
\label{tab:1-1}
\end{table}

\begin{figure}[hbtp]
\centering
{\unitlength 0.1in%
\begin{picture}(14.7000,6.3000)(1.3000,-10.3000)%
%
\special{pn 8}%
\special{pa 400 400}%
\special{pa 1000 1000}%
\special{fp}%
\special{pa 1000 1000}%
\special{pa 1600 400}%
\special{fp}%
\special{pa 1600 1000}%
\special{pa 400 1000}%
\special{fp}%
\special{pa 460 460}%
\special{pa 520 400}%
\special{fp}%
\special{pa 520 400}%
\special{pa 640 520}%
\special{fp}%
\special{pa 640 520}%
\special{pa 700 460}%
\special{fp}%
\special{pa 700 460}%
\special{pa 760 520}%
\special{fp}%
\special{pa 760 520}%
\special{pa 820 460}%
\special{fp}%
\special{pa 820 460}%
\special{pa 940 580}%
\special{fp}%
\special{pa 940 580}%
\special{pa 1000 520}%
\special{fp}%
\special{pa 1000 520}%
\special{pa 1060 580}%
\special{fp}%
\special{pa 1060 580}%
\special{pa 1180 460}%
\special{fp}%
\special{pa 1180 460}%
\special{pa 1300 580}%
\special{fp}%
\special{pa 1300 580}%
\special{pa 1360 520}%
\special{fp}%
\special{pa 1360 520}%
\special{pa 1420 580}%
\special{fp}%
%
\special{pn 8}%
\special{pa 1420 580}%
\special{pa 1420 1000}%
\special{dt 0.045}%
\special{pa 1360 1000}%
\special{pa 1360 520}%
\special{dt 0.045}%
\special{pa 460 460}%
\special{pa 460 1000}%
\special{dt 0.045}%
\special{pa 520 400}%
\special{pa 520 1000}%
\special{dt 0.045}%
\special{pa 640 1000}%
\special{pa 640 520}%
\special{dt 0.045}%
\put(4.3000,-10.2000){\makebox(0,0)[rt]{$x_1$}}%
\put(4.9000,-10.2000){\makebox(0,0)[lt]{$y_1$}}%
\put(6.6000,-10.3000){\makebox(0,0)[lt]{$x_2$}}%
\put(13.6000,-10.3000){\makebox(0,0)[rt]{$y_{r-1}$}}%
\put(14.2000,-10.3000){\makebox(0,0)[lt]{$x_r$}}%
\end{picture}}%
\caption{Coordinates of a rectanglular diagram}
\label{fig:1-1}
\end{figure}
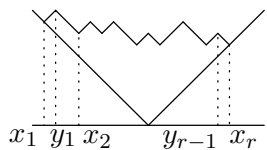

Diagram $\lambda$ is displayed in the $xy$ plane as Figure~\ref{fig:1-1}. 
Its profile is identified with the graph $y= \lambda(x)$, where 
the coordinates of valleys $x_i$ and peaks $y_i$ belong to $\mathbb{Z}$ and satisfy 
\begin{equation}\label{eq:1-3}
x_1< y_1< x_2< \cdots < y_{r-1} < x_r, \qquad \sum_i x_i = \sum_i y_i.
\end{equation}
We then have 
\[
\int_\mathbb{R} ( \lambda(x) - \lvert x\rvert) dx = 2 \lvert \lambda\rvert.
\]

Regardless of being integers, we consider a rectangular diagram which has valleys and 
peaks satisfying \eqref{eq:1-3}. 
More generally, an element of 
\begin{multline*}
\mathscr{D} = \bigl\{ \omega : \mathbb{R}\longrightarrow \mathbb{R} \,\big\vert\, 
\lvert \omega(x)-\omega(y)\rvert \leqq \lvert x-y\rvert, \\ 
\int_{-\infty}^{-1}(1+\omega^\prime (x)) \frac{dx}{\lvert x\rvert} <\infty, \; 
\int_1^\infty (1-\omega^\prime (x)) \frac{dx}{x} <\infty\bigr\}
\end{multline*}
is called a continuous diagram. 
A probability on $\mathbb{R}$ is assigned to $\omega\in\mathscr{D}$, 
denoted by $\mathfrak{m}_\omega$, and called Kerov's transition measure of $\omega$ 
(\cite{Ker03}). 
If $\lambda$ is a rectangular diagram satisfying \eqref{eq:1-3}, $\mathfrak{m}_\lambda$ 
is an atomic probability supported by $\{x_1, \cdots, x_r\}$ and characterized by 
\[
\frac{(z-y_1)\cdots (z-y_{r-1})}{(z-x_1)\cdots (z-x_r)} = 
\sum_{i=1}^r \frac{\mathfrak{m_\lambda}(\{x_i\})}{z-x_i}.
\] 
Furthermore, if $\lambda\in\mathbb{Y}$, then 
$\mathrm{supp}\,\mathfrak{m}_\lambda = \{ c(\mu\setminus\lambda) \,\vert\, 
\lambda\nearrow\mu \}$ 
by using content $c$ of a box. 
The $k$th moment of probability $\mathfrak{m}$ on $\mathbb{R}$ is denoted by 
$M_k(\mathfrak{m})$ if it exists. 
The moment sequence $\{ M_k(\mathfrak{m}_\omega)\}_{k\in\mathbb{N}}$ 
determines $\omega\in\mathscr{D}$ if the moment problem is determinate. 

Let $\chi^\lambda$ denote the irreducible character of $\mathfrak{S}_n$ 
corresponding to $\lambda\in\mathbb{Y}_n$. 
Its value at $x\in\mathfrak{S}_n$ is denoted by $\chi^\lambda(x)$ or $\chi^\lambda_C$ 
when $x$ belongs to conjugacy class $C$ of $\mathfrak{S}_n$. 
For conjugacy class $C_\sigma$ corresponding to $\sigma\in\mathbb{Y}_n$, we 
simply write as $\chi^\lambda_\sigma$. 
Naturally, $\chi^\lambda$ is linearly extended onto $\mathbb{C}[\mathfrak{S}_n]$. 
Define restriction map 
$\mathrm{E} : \mathbb{C}[\mathfrak{S}_{n+1}] \longrightarrow \mathbb{C}[\mathfrak{S}_n]$ 
by 
\begin{equation}\label{eq:1-6}
\mathrm{E}[x] := \begin{cases} x, & x\in\mathfrak{S}_n \\ 0, & 
x\in\mathfrak{S}_{n+1}\setminus\mathfrak{S}_n \end{cases}
\end{equation}
with linear extension. 
Biane's formula takes the form of 
(see \cite{Bia03}, \cite{Hor05}, \cite{Hor06}) 
\begin{equation}\label{eq:1-7}
\frac{\chi^\lambda(\mathrm{E}[ J_{n+1}^{\ k}])}{\dim\lambda} 
= M_k (\mathfrak{m}_\lambda) \quad 
\Bigl( = \sum_{\mu: \lambda\nearrow\mu} c(\mu\setminus\lambda)^k 
\mathfrak{m}_\lambda(\{c(\mu\setminus\lambda)\}) \Bigr), \qquad 
\lambda\in\mathbb{Y}_n, \ k\in\mathbb{N}.
\end{equation}

JM elements and Biane's formula appear also in \cite{KvLiMi19} related to free probability. 
JM elements admit a spin version (\cite{Kle05}). 
In \cite{Hor24}, we showed a spin version of \eqref{eq:1-7} and applied it to spin 
limit shape problems for symmetric groups. 

It is natural to consider extension from a symmetric group to wreath product 
\[ 
\mathfrak{S}_n(T) = T \wr \mathfrak{S}_n = T^n \rtimes \mathfrak{S}_n
\] 
where $T$ is a finite group. 
The set of conjugacy classes of $T$ is denoted by $[T]$. 
An element of $\mathfrak{S}_n(T)$ is decomposed into product of disjoint 
cycles of $\mathfrak{S}_n$, e.g. $(i_1\ i_2\ \cdots\ i_p)$, with elements of $T$, 
e.g. $(t_{i_1}, t_{i_2}, \cdots, t_{i_p})$, supported by them. 
Such cycles are conjugate in $\mathfrak{S}_n(T)$ if the cycle lengths agree and 
the products $t_{i_p}\cdots t_{i_1}$ belong to the same conjugacy class of $T$. 
This implies that the conjugacy classes $[\mathfrak{S}_n(T)]$ are parametrized by 
\[ 
\mathbb{Y}_n([T]) := \bigl\{ \rho = (\rho_\theta)_{\theta\in [T]} \,\big\vert\, 
\rho_\theta\in\mathbb{Y}, \ \sum_{\theta\in [T]} \lvert\rho_\theta\rvert =n 
\bigr\}.
\] 
If $g$ belongs to conjugacy class $C_\rho$ labeled by $\rho\in\mathbb{Y}_n([T])$, 
then $\rho$ is referred to as the (conjugacy) type of $g$. 
On the other hand, we see the equivalence classes of irreducible representations 
$\mathfrak{S}_n(T)^\wedge$ are parametrized by 
\[ 
\mathbb{Y}_n(\widehat{T}) := \bigl\{ \lambda = (\lambda^\zeta)_{\zeta\in \widehat{T}} 
\,\big\vert\, \lambda^\zeta\in\mathbb{Y}, \ \sum_{\zeta\in \widehat{T}} 
\lvert\lambda^\zeta\rvert =n \bigr\}
\] 
by the construction using induction from its subgroups. 

Similarly to \eqref{eq:1-1}, a JM element of a wreath product is defined as 
difference of two conjugacy class indicators of $2$-cycles. 
However, $\mathfrak{S}_n(T)$ has $\lvert [T]\rvert$ numbers of $2$-cycle conjugacy classes. 
In the so-called dual approach of asymptotic representation theory, we often discuss 
behavior of some quantities by letting the type of a conjugacy class be fixed and the size of 
a group grow infinitely. 
It is then convenient to consider an inclusion map 
$\iota_{n,k} : \mathbb{Y}_k([T]) \longrightarrow \mathbb{Y}_n([T])$ for 
$k, n\in\mathbb{N}$, $k<n$, where 
\[ 
\iota_{n,k}\rho = (\Tilde{\rho}_\theta)_{\theta\in [T]} \in \mathbb{Y}_n([T]), 
\qquad \Tilde{\rho}_\theta := \begin{cases} \rho_{\{e_T\}} \sqcup (1^{n-k}), & 
\theta = \{e_T\} \\ \rho_\theta, & \theta\neq\{e_T\} \end{cases}
\] 
for $\rho = (\rho_\theta)_{\theta\in [T]} \in \mathbb{Y}_k([T])$. 
Let $(2)_\theta$ denote the element of $\mathbb{Y}_2([T])$ whose $\theta$-entry is 
cycle $(2)$ with the other entries $\varnothing$. 
The conjugacy class of $\mathfrak{S}_n(T)$ having 
$\iota_{n,2}(2)_\theta \in\mathbb{Y}_n([T])$ as its type, i.e. the set of 
$(\theta, 2)$-cycles, is 
\[ 
C_{\iota_{n,2}(2)_\theta} = \{ (y^{-1}, (i)) (xy, (j)) (i\ j) \,\vert\, 
1\leqq i<j\leqq n, \ y\in T, \ x\in C_\theta\}.
\] 
Here $C_\theta$ is the conjugacy class of $T$ labeled by $\theta$, and $(y^{-1}, (i))$ 
denotes the element of $T^n$ whose $i$th entry is $y^{-1}$ with the other entries $e_T$. 
Let $A_C$ denote the indicator of conjugacy class $C$ of a group
\footnote{Notation used throughout the present paper.}: 
$A_C := \sum_{g\in C} g$. 
For $\theta\in [T]$ and $n\in\mathbb{N}$, we define JM element 
$J_{n+1}^\theta$ of wreath product $\mathfrak{S}_{n+1}(T)$ as 
the difference of $(\theta, 2)$-cycles of 
$\mathfrak{S}_{n+1}(T)$ and those of $\mathfrak{S}_n(T)$ in 
$\mathbb{C}[\mathfrak{S}_{n+1}([T])]$: 
\begin{equation}\label{eq:1-13}
J_{n+1}^\theta := A_{\iota_{n+1,2}(2)_\theta} - A_{\iota_{n,2}(2)_\theta} 
= \sum_{i=1}^n \sum_{x\in C_\theta} \sum_{y\in T} (y^{-1}, (i)) (xy, (n\!+\!1)) (i\ n\!+\!1). 
\end{equation}
When $\theta = \{e_T\}$, $J_{n+1}^{\{e_T\}}$ is simply denoted by $J_{n+1}$. 
We see straightforwardly 
\[ 
J_{n+1}^\theta = J_{n+1} (A_\theta, (n\!+\!1)) = (A_\theta, (n\!+\!1)) J_{n+1}
\] 
where $A_\theta := A_{C_\theta} = \sum_{x\in C_\theta} x$, so 
$(A_\theta, (n\!+\!1)) = \sum_{x\in C_\theta} (x, (n\!+\!1))$. 

The following result is one of the theorems in this paper, which extends Biane's formula 
\eqref{eq:1-7} to wreath products. 
Under the inclusion $\mathfrak{S}_n(T) \subset \mathfrak{S}_{n+1}(T)$ through 
$T^n \cong T^n\times \{e_T\} \subset T^{n+1}$, the restriction map 
$\mathrm{E} : \mathbb{C}[\mathfrak{S}_{n+1}(T)] \longrightarrow 
\mathbb{C}[\mathfrak{S}_n(T)]$ is defined similarly to \eqref{eq:1-6}. 
The irreducible character of $\mathfrak{S}_n(T)$ corresponding to 
$\lambda\in\mathbb{Y}_n(\widehat{T})$ is denoted by $\chi^\lambda$. 
The character table of $T$ is denoted by $(\chi^\zeta_\theta)_{\zeta, \theta}$. 

\begin{theorem}\label{th:1-1}
For $\lambda = (\lambda^\zeta)_{\zeta\in \widehat{T}} \in\mathbb{Y}_n(\widehat{T})$ 
and $k\in\mathbb{N}$, we have 
\begin{equation}\label{eq:1-15}
M_k(\mathfrak{m}_{\lambda^\zeta}) = \frac{(\dim\zeta)^{k-1}}{\lvert T\rvert^k} 
\sum_{\theta\in [T]} \frac{\chi^\lambda ( \mathrm{E}[ J_{n+1}^{\ k} (A_\theta, (n\!+\!1))])}
{\dim\lambda} \, \overline{\chi}^\zeta_\theta, \qquad \zeta\in \widehat{T}, 
\end{equation}
or equivalently (by inversion) 
\begin{equation}\label{eq:1-16}
\frac{\chi^\lambda ( \mathrm{E}[ J_{n+1}^{\ k} (A_\theta, (n\!+\!1))])}{\dim\lambda} 
= \lvert T\rvert^{k-1} \lvert C_\theta\rvert 
\sum_{\zeta\in\widehat{T}} \frac{M_k(\mathfrak{m}_{\lambda^\zeta}) \chi^\zeta_\theta}
{(\dim\zeta)^{k-1}}, \qquad \theta\in [T].
\end{equation}
\end{theorem}

Proof of Theorem~\ref{th:1-1} is given in \S2. 
In the case of $T = \mathbb{Z}_2$, a formula corresponding to \eqref{eq:1-16} was 
reported in \cite{Hor05a}. 

We are interested in asymptotic estimate of 
$\mathrm{E}[ J_{n+1}^{\ k} (A_\theta, (n\!+\!1))]$ in \eqref{eq:1-15} and \eqref{eq:1-16}, 
especially by a combinatorial method. 
Noting $\mathrm{E}[ J_{n+1}^{\ k} (A_\theta, (n\!+\!1))]$ belongs to the center of 
$\mathbb{C}[\mathfrak{S}_n(T)]$ and is expressed as a linear combination of conjugacy class 
indicators, we want to find asymptotic behavior of its coefficients as $n\to\infty$. 
We assemble some (non-standard) notations in Table~\ref{tab:1-1.5}. 

\begin{table}[htb]
\begin{tabular}{l} \toprule 
$\overline{\mathbb{Y}} := \{\sigma\in \mathbb{Y} \,\vert\, m_1(\sigma) =0\}$, \qquad 
$\overline{\mathbb{Y}}_k := \{ \sigma\in \overline{\mathbb{Y}} \,\vert\, |\sigma|=k\}$,\\ 
$\sigma^\times \in \mathbb{Y}$ : remove the first column of $\sigma\in \overline{\mathbb{Y}}$ 
\quad i.e. \  $m_j(\sigma) = m_{j-1}(\sigma^\times)$ for $j\geqq 2$, \\ 
$\mathrm{NC}(k)$ : the set of non-crossing partitions of $\{ 1,2,\dots, k\}$, \\ 
$\mathrm{NC}(\sigma) := \{ \pi\in\mathrm{NC}(k) \,\vert\, \text{type of blocks of }\pi 
\text{ is }\sigma\}$ \qquad ($\sigma\in\mathbb{Y}_k$), \\ 
multinomial coefficient : 
$\binom{m}{(n^\alpha)_{\alpha \in A}} := \frac{m!}{\prod_{\alpha\in A} n^\alpha !}$ 
\ where \ $\sum_{\alpha\in A} n^\alpha = m$ \ ($\lvert A\rvert <\infty$) \\ 
\bottomrule 
\end{tabular}
\caption{Some notations}
\label{tab:1-1.5}
\end{table}

\begin{figure}[hbtp]
\centering
{\unitlength 0.1in%
\begin{picture}(18.3000,4.0500)(5.7000,-8.0500)%
%
\special{pn 8}%
\special{pa 600 800}%
\special{pa 600 400}%
\special{fp}%
\special{pa 800 800}%
\special{pa 800 600}%
\special{fp}%
\special{pa 1000 800}%
\special{pa 1000 600}%
\special{fp}%
\special{pa 1200 600}%
\special{pa 1200 800}%
\special{fp}%
\special{pa 1400 800}%
\special{pa 1400 600}%
\special{fp}%
\special{pa 1400 600}%
\special{pa 800 600}%
\special{fp}%
\special{pa 1600 800}%
\special{pa 1600 400}%
\special{fp}%
\special{pa 1800 800}%
\special{pa 1800 600}%
\special{fp}%
\special{pa 2000 600}%
\special{pa 2000 800}%
\special{fp}%
\special{pa 2000 600}%
\special{pa 1800 600}%
\special{fp}%
\special{pa 2200 800}%
\special{pa 2200 400}%
\special{fp}%
\special{pa 2400 400}%
\special{pa 2400 800}%
\special{fp}%
\special{pa 2400 400}%
\special{pa 600 400}%
\special{fp}%
\put(6.0000,-8.7000){\makebox(0,0){1}}%
\put(8.0000,-8.7000){\makebox(0,0){2}}%
\put(10.0000,-8.7000){\makebox(0,0){3}}%
\put(12.0000,-8.7000){\makebox(0,0){4}}%
\put(14.0000,-8.7000){\makebox(0,0){5}}%
\put(16.0000,-8.7000){\makebox(0,0){6}}%
\put(18.0000,-8.7000){\makebox(0,0){7}}%
\put(20.0000,-8.7000){\makebox(0,0){8}}%
\put(22.0000,-8.7000){\makebox(0,0){9}}%
\put(24.0000,-8.7000){\makebox(0,0){10}}%
\end{picture}}%
\caption{$\pi = \{ \{ 1,6,9,10\}, \{2,3,4,5\}, \{7,8\}\} \in 
\mathrm{NC}((2^14^2)) \subset \mathrm{NC}(10)$}
\label{fig:1-2}
\end{figure}
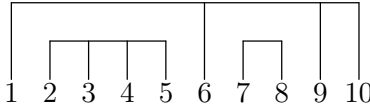

\begin{theorem}\label{th:1-2}
Let $T$ be abelian
\footnote{We have not succeeded yet in obtaining subsequent theorems for non-abelian $T$.}. 
Set $T^\circ := T\setminus \{e_T\}$. 
For $y\in T$ and $k\in\mathbb{N}$, we have 
\begin{align}
\mathrm{E}[ J_{n+1}^{\ k} (y, (n+1))] = &\sum_{\sigma\in\overline{\mathbb{Y}}_k} 
\Bigl\{ \sum_{(n^x_1, n^x_2, \dots)_{x\in T^\circ} : (\star)_y} 
\lvert \mathrm{NC}(\sigma)\rvert n^{k-l(\sigma)} \bigl( 1+ O(n^{-1})\bigr) 
\lvert T\rvert^{k-l(\sigma)} 
\notag \\ 
&\cdot\binom{m_1(\sigma^\times)}{(n^x_1)_{x\in T}}
\binom{m_2(\sigma^\times)}{(n^x_2)_{x\in T}}
\cdots \ 
\frac{A_{\iota_{n, \lvert\sigma^\times\rvert} ((1^{n^x_1} 2^{n^x_2} \cdots))_{x\in T}}}
{\lvert C_{\iota_{n, \lvert\sigma^\times\rvert} ((1^{n^x_1} 2^{n^x_2} \cdots))_{x\in T}}\rvert} 
+ \sum_{g\in N_\sigma} g\Bigr\} \notag \\ 
&\qquad\qquad\qquad\qquad\qquad\qquad\qquad\qquad\qquad (n\to\infty).
\label{eq:1-17}
\end{align}
Here $(\star)_y$, the range over which $(n^x_1, n^x_2, \dots)_{x\in T^\circ}$ 
runs, is given by
\footnote{$n^x_i$ indicates the number of rows of length $i$ in $\sigma^\times$ 
labeled by $x$. 
The rows not labeled by any $x\in T^\circ$ are regarded to be labeled by $e_T$. 
} 
\begin{equation}\label{eq:1-18}
n^x_i \in \{ 0, 1, \dots, m_i(\sigma^\times)\}, \ 
\sum_{x\in T^\circ} n^x_i \leqq m_i(\sigma^\times)
\  (i=1,2, \dots), \quad 
\text{and} \quad \prod_{x\in T^\circ} x^{n^x_1+n^x_2+\cdots } = y. 
\end{equation}
Furthermore, set 
\begin{equation}\label{eq:1-19}
n^{e_T}_i := m_i(\sigma^\times) - \sum_{x\in T^\circ} n^x_i.
\end{equation}
The extra terms $N_\sigma$ consist of some of those elements of $\mathfrak{S}_n(T)$ whose 
types have $\sigma^\times$, or maybe missing some $1$-box rows, as nontrivial cycles, and satisfy 
\begin{equation}\label{eq:1-20}
\lvert N_\sigma \rvert = O( n^{k- l(\sigma)-1}) \qquad (n\to\infty). 
\end{equation}
\end{theorem}

Since $m_i(\sigma^\times) =0$ hence $n^x_i =0$ holds for large $i$, the range \eqref{eq:1-18} 
of sum in \eqref{eq:1-17} is finite and does not depend on $n$.
Proof of Theorem~\ref{th:1-2} is given in \S3. 
Equation \eqref{eq:1-17} in the case of symmetric groups was shown in \cite[Theorem~1]{Hor06}. 

So far we exploited some properties of JM elements of wreath products. 
In the other part of this paper, we apply them to a \textit{dynamical} limit shape problem for 
random \textit{multiple} Young diagrams. 
The model treated here is the same with the continuous time random walk on 
$\mathbb{Y}_n(\widehat{T})$ formulated in \cite{Hor25}. 
Let us quickly review it. 
Given $\lambda\in \mathbb{Y}_n(\widehat{T})$ and a corresponding irreducible 
representation $\pi^\lambda$ of $\mathfrak{S}_n(T)$, consider irreducible decomposition 
of its restriction and induction: 
\begin{equation}\label{eq:1-21}
\mathrm{Ind}_{\mathfrak{S}_{n-1}(T)}^{\mathfrak{S}_n(T)} \circ 
\mathrm{Res}_{\mathfrak{S}_{n-1}(T)}^{\mathfrak{S}_n(T)} \pi^\lambda \cong 
\bigoplus_{\nu\in \mathbb{Y}_{n-1}(\widehat{T}) : \nu\nearrow\lambda} \ 
\bigoplus_{\mu\in \mathbb{Y}_n(\widehat{T}) : \nu\nearrow\mu} 
[\dim \zeta_{\nu\lambda} \dim\zeta_{\nu\mu}] \pi^\mu.
\end{equation}
Here we use notation $\nu\nearrow\lambda$ for $\nu\in\mathbb{Y}_{n-1}(\widehat{T})$ 
and $\lambda\in\mathbb{Y}_n(\widehat{T})$ to say that a box is added at one entry 
of $\nu$ to get $\lambda$,  
where $\zeta_{\nu\lambda}$ indicates the entry at which a box is added, hence 
$\nu^{\zeta_{\nu\lambda}}\nearrow\lambda^{\zeta_{\nu\lambda}}$ as 
diagrams (see Table~\ref{tab:1-1}). 
Taking dimensions of both sides in \eqref{eq:1-21} and dividing them by 
$n \lvert T\rvert \dim\lambda$, we get a stochastic matrix 
\begin{equation}\label{eq:1-22}
P = P^{(n)} = (P_{\lambda\mu})_{\lambda, \mu\in \mathbb{Y}_n(\widehat{T})}, \qquad 
P_{\lambda\mu} = \sum_{\nu\in\mathbb{Y}_{n-1}(\widehat{T}):\, 
\nu\nearrow\lambda, \nu\nearrow\mu} \dim\zeta_{\nu\lambda} \dim\zeta_{\nu\mu} \,
\frac{\dim\mu}{n \lvert T\rvert \dim\lambda}.
\end{equation}
Markov chain $(Z^{(n)}_k)_{k\in\{0,1,2,\dots\}}$ caused by the transition matrix \eqref{eq:1-22}, 
called a Res-Ind chain, on $\mathbb{Y}_n(\widehat{T})$ is symmetric with respect to 
the Plancherel measure 
\[ 
M_{\mathrm{Pl}} (\lambda) = \frac{(\dim\lambda)^2}{n! \lvert T\rvert^n}, 
\qquad \lambda\in \mathbb{Y}_n(\widehat{T}).
\] 
We turn $(Z^{(n)}_k)_k$ into a continuous time process $(X^{(n)}_s)_{s\geqq 0}$ 
by introducing pausing (or holding) time. 
Let IID pausing times obey distribution $\psi$ on $(0, \infty)$. 
This produces counting process $(N_s)_{s\geqq 0}$, where $N_0 \equiv 0$ a.s. 
Then, $(X^{(n)}_s)_{s\geqq 0}$ is defined by 
\begin{equation}\label{eq:1-24}
X^{(n)}_s := Z^{(n)}_{N_s}, \qquad s\geqq 0. 
\end{equation}
In \cite{Hor25}, considering appropriate space-time scaling limits for this process, we 
derived time evolution of \textit{averaged} limit shapes of multi-diagrams. 
In the present paper, we exploit finer estimates and obtain time evolution of 
\textit{concentrated} limit shapes of multi-diagrams (in the sense of LLN) as Theorem~\ref{th:1-3} 
though we restrict ourselves to the case of abelian $T$. 

In order to obtain concentrated properties, the notion of approximate factorization property 
is useful. 
This notion is introduced by Biane in \cite{Bia01} for symmetric groups and closely related 
to ergodicity of a probability measure. 
Let us quickly review it in the context of wreath products. 
The set of $\mathbb{C}$-valued normalized positive-definite central functions $f$ on 
$\mathfrak{S}_n (T)$ has affine bijective correspondence to the set of probabilities $M$ 
on $\mathbb{Y}_n(\widehat{T})$ through the Fourier transform 
\begin{equation}\label{eq:1-25}
f(g) = \sum_{\lambda\in\mathbb{Y}_n(\widehat{T})} 
M(\lambda) \,\frac{\chi^\lambda(g)}{\dim\lambda}, \qquad g\in \mathfrak{S}_n(T).
\end{equation}
This affine bijection is lifted to the one between the $\mathbb{C}$-valued normalized 
positive-definite central functions on $\mathfrak{S}_\infty (T)$, 
the infinite wreath product of $T$ by the infinite symmetric group, and the central probabilities 
on the path space of the branching graph for $\mathfrak{S}_\infty (T)$ where \eqref{eq:1-25} 
gives the relation between the restriction onto $\mathfrak{S}_n(T)$ and the marginal 
distribution on $\mathbb{Y}_n(\widehat{T})$. 
See \cite{HoHiHi08} and \cite{HoHi14} for some roles played by this correspondence in harmonic 
analysis. 
In particular, we have correspondence between the sets of extremal points of both sides. 
Recall that, for $g= (x_1, \cdots, x_n)\tau\in\mathfrak{S}_n(T)$ ($x_j\in T$, 
$\tau\in\mathfrak{S}_n$), $\mathrm{supp}\,g$ consisits of those $j$'s such that 
$\tau (j)\neq j$ or $x_j\neq e_T$. 
It is naturally extended to $g\in\mathfrak{S}_\infty (T)$.  
Such a function $f$ on $\mathfrak{S}_\infty (T)$ is extremal if and only if it is 
factorizable: 
\[ 
f(g h) = f(g)f(h), \qquad \text{if} \quad g, h\in \mathfrak{S}_\infty (T) \ \text{ and } 
\ \mathrm{supp}\,g \cap \mathrm{supp}\,h = \varnothing
\] 
(see \cite{HiHi05}). 
For a probability, \lq extremal\rq\ is often referred to as \lq ergodic\rq. 
Ergodicity is a typical structure which gives rise to concentration phenomena for probabilities, 
such as LLN. 
The notion of ergodicity can be weakened in various ways according to situations. 
In terms of corresponding function $f$, that means relaxiation of factorizability so that 
$f\vert_{\mathfrak{S}_n}$ factorizes approximately in some order as $n\to\infty$. 
Let $\mathrm{type}(g)\in \mathbb{Y}([T])$ denote the cycle type of 
$g\in \mathfrak{S}_\infty(T)$, where we set 
\begin{equation}\label{eq:1-26.5}
\mathbb{Y}([T]) = \bigsqcup_{k=0}^\infty \{ \rho\in \mathbb{Y}_k([T]) \,\vert\, 
m_1(\rho_{\{e_T\}}) =0\}.
\end{equation}
By definition, taking $k\in\mathbb{N}$ such that $g\in \mathfrak{S}_k(T)$ and conjugacy 
class $C_\rho$ of $\mathfrak{S}_k(T)$ ($\rho\in\mathbb{Y}_k([T])$) to which $g$ belongs, 
\textit{we get $\mathrm{type}(g)$ by removing all rows of length $1$ from the 
$\{e_T\}$-entry of $\rho$}. 
In particular, $\mathrm{type}(e) = (\varnothing, \cdots, \varnothing)$. 
The (total) size and length of $\rho = (\rho_\theta)_{\theta\in [T]}$ is denoted by 
\[ 
\lvert\rho\rvert := \sum_{\theta\in [T]} \lvert\rho_\theta \rvert, \qquad 
l(\rho) := \sum_{\theta\in [T]} l(\rho_\theta).
\] 

\medskip

\noindent\textbf{Definition} \ 
Let $\{f^{(n)}\}_{n\in\mathbb{N}}$ be a sequence of $\mathbb{C}$-valued normalized 
positive-definite central functions, each on $\mathfrak{S}_n(T)$. 
Let $M^{(n)}$ be the probability on $\mathbb{Y}_n(\widehat{T})$ related to $f^{(n)}$ 
by \eqref{eq:1-25}. 
\textit{Approximate factorization property} (AFP, for short) for $\{f^{(n)}\}_{n\in\mathbb{N}}$ 
or $\{M^{(n)}\}_{n\in\mathbb{N}}$ is defined by the following condition: 
if $\mathrm{supp}\,g \cap \mathrm{supp}\,h = \varnothing$ for 
$g, h\in \mathfrak{S}_\infty(T)$, 
\begin{equation}\label{eq:1-28}
f^{(n)}(gh) = f^{(n)}(g) f^{(n)}(h) + o\bigl( n^{-\frac{1}{2}
(\lvert\mathrm{type}(g)\rvert - l(\mathrm{type}(g)) + 
\lvert\mathrm{type}(h)\rvert - l(\mathrm{type}(h)))} \bigr) \  (n\to\infty) 
\end{equation}
holds. 
Here $f^{(n)}(gh) - f^{(n)}(g) f^{(n)}(h)$ depends only on types of $g$ and $h$ 
(rather than $g$ and $h$ themselves).
\hfill $\square$

\medskip

The above definition of AFP is probably the weakest version. 
See \cite{Sni06} for applications of AFP to central limit theorems in wreath products. 

If $g, h$ are $k, l$-cycles respectively, the $o$ terms in \eqref{eq:1-28} is 
$o(n^{-\frac{1}{2}(k-1+l-1)})$. 
Under \eqref{eq:1-31} below, $f^{(n)}(g)$ and $f^{(n)}(h)$ are of orders 
$n^{-\frac{k-1}{2}}$ and $n^{-\frac{l-1}{2}}$ respectively. 
Hence \eqref{eq:1-28} expresses factrorization with smaller error terms. 

Let us consider a space-time scaling limit for process $(X^{(n)}_s)_{s\geqq 0}$ of 
\eqref{eq:1-24}. 
In order to get a limit shape of a diagram, we take $1/\sqrt{n}$ in the spatial direction: 
for $y = \nu(x)$ in Figure~\ref{fig:1-1}, 
\begin{equation}\label{eq:1-29}
\nu^{\sqrt{n}}(x) := \frac{1}{\sqrt{n}} \,\nu (\sqrt{n} x), \qquad 
\nu\in\mathbb{Y}_n.
\end{equation}
In the temporal direction, we take square of $\sqrt{n}$, namely a diffusive rescale
\footnote{Other rescales in \cite{Hor25} are irrelevant here. 
See Remark (4) after Theorem~\ref{th:1-3}.}: 
\[ 
s = tn \qquad \text{for} \quad t\geqq 0,
\] 
where $s$ and $t$ are interpreted as microscopic and macroscopic times respectively. 
So let $M_{tn}^{(n)}$ be the distribution of $X^{(n)}_{tn}$ on $\mathbb{Y}_n(\widehat{T})$, 
and $f^{(n)}_{tn}$ the function related to $M^{(n)}_{tn}$ by \eqref{eq:1-25}. 
We show time evolution of concentration phenomenon for our model under the assumptions on 
the initial condition for $M^{(n)}_0$ or $f^{(n)}_0$ and the pausing time distribution.

\begin{theorem}\label{th:1-3}
Let $T$ be abelian
\footnote{The paring of $x\in T$ and $\zeta\in\widehat{T}$ is denoted by $\zeta(x)$.}. 
As an initial condition, assume for $\{f^{(n)}_0\}_{n\in\mathbb{N}}$ 

\noindent 
convergence of $(x, k)$-cycle values and growth orders:
\begin{align}
&\text{there exists} \quad \lim_{n\to\infty} n^{\frac{k-1}{2}} f^{(n)}_0 
( \iota_{n,k} (k)_x) =: \gamma_{k+1}^x, 
& &x\in T, \ k\in\mathbb{N}, \label{eq:1-31} \\ 
&\sup_{k\geqq 2} \frac{\lvert \gamma_k^x\rvert^{1/k}}{k} < \infty, & 
&x\in T, \label{eq:1-32}
\end{align}
and AFP \eqref{eq:1-28}. 
For pausing time distribution $\psi$, assume that $\psi$ has mean $m$ and its characteristic 
function $\varphi$ satisfies: 
\begin{equation}\label{eq:1-33}
\text{there exists} \quad a>0, \qquad \int_{\{ \lvert\xi\rvert \geqq a\}} 
\Bigl\lvert \frac{\varphi(\xi)}{\xi}
\Bigr\rvert d\xi < \infty.
\end{equation}
Then, for any $t\geqq 0$, we have concentration (LLN) for $M^{(n)}_{tn}$ as $n\to\infty$, 
that is, there exists $\omega(t) = (\omega(t)^\zeta)_{\zeta\in\widehat{T}}$ 
such that $\omega(t)^\zeta \in \mathscr{D}$ satisfies, for any \ $\epsilon >0$, 
\begin{equation}\label{eq:1-34}
\lim_{n\to\infty} M^{(n)}_{tn} \Bigl( \Bigl\{ \lambda = (\lambda^\zeta)_{\zeta\in\widehat{T}}
\in \mathbb{Y}_n(\widehat{T}) \,\Big\vert\, \max_{\zeta\in\widehat{T}} 
\lVert (\lambda^\zeta)^{\sqrt{n}}- \omega(t)^\zeta \rVert_{\sup} > \epsilon\Bigr\}\Bigr) =0.
\end{equation}
The limit shape $\omega(t)$ is described in terms of free cumulants of transition measures 
$\mathfrak{m}_{\omega(t)^\zeta}$, $\zeta\in \widehat{T}$, as 
\begin{align}
&R_1 (\mathfrak{m}_{\omega(0)^\zeta}) =0, & 
&R_{k+1}(\mathfrak{m}_{\omega(0)^\zeta}) = \frac{1}{\lvert T\rvert} \sum_{x\in T} 
\gamma_{k+1}^x \overline{\zeta(x)} \qquad (k\in\mathbb{N}); 
\label{eq:1-35} \\ 
&R_1 (\mathfrak{m}_{\omega(t)^\zeta}) =0, & 
&R_2 (\mathfrak{m}_{\omega(t)^\zeta}) = \frac{1- e^{-\frac{t}{m}}}{\lvert T\rvert} 
+ e^{-\frac{t}{m}} R_2 (\mathfrak{m}_{\omega(0)^\zeta}), \notag \\ 
&\quad & &R_{k+1}(\mathfrak{m}_{\omega(t)^\zeta}) = 
e^{-\frac{kt}{m}} R_{k+1}(\mathfrak{m}_{\omega(0)^\zeta}) \qquad (k\geqq 2).
\label{eq:1-36}
\end{align}
\end{theorem}

Proof of Theorem~\ref{th:1-3} is given in \S4. 

\medskip

\begin{table}[htb]
\begin{tabular}{lll} \toprule 
& reference \cite{Hor25} & present paper \\ 
\midrule 
limit shape & averaged & concentrated \\ 
finite group for wreath product & general & mainly, abelian \\ 
pausing time distribution & exponential-like, stable & exponential-like \\ 
JM element & not used & used \\ 
AFP & not used & used \\ 
\bottomrule 
\end{tabular}
\caption{Differences between two companion papers}
\label{tab:1-2}
\end{table}

\noindent\textbf{Remark} \ 
(1) Rescale in \eqref{eq:1-31} seems to be natural if one considers normalized irreducible 
character $\chi^{(n)}$ of $\mathfrak{S}_n$. 
In the regime having a limit under rescale of \eqref{eq:1-29}, it is known that 
$n^{\downarrow k} \chi^{(n)}(k\text{-cycle})$ is as large as ($k+1$)th cumulant or moment 
of Kerov's transition measure, and hence is of order $n^{\frac{k+1}{2}}$. 
This is consistent to \eqref{eq:1-31}. 

(2) Relative size of each entry at any time $t$ is connected with variance 
(the second free cumulant) of transition measure by 
\[ 
R_2( \mathfrak{m}_{\omega(t)^\zeta}) = \lim_{n\to\infty} 
\mathbb{E}_{M^{(n)}_{tn}} \Bigl[ \frac{\lvert\lambda^\zeta\rvert}{n} \Bigr], 
\qquad t\geqq 0.
\] 

(3) The relation \eqref{eq:1-36} for free cumulants can be rewritten as 
\[ 
\mathfrak{m}_{\omega(t)^\zeta} = (\mathfrak{m}_{\omega(0)^\zeta})_{e^{-t/m}} 
\boxplus \mathfrak{c}_{1-e^{-t/m}} 
\] 
where $\mathfrak{c}$ is the semi-circle distribution with mean $0$ and variance $1/\lvert T\rvert$, 
$\boxplus$ denotes the free additive convolution, and subscription $(\ )_c$ denotes 
the free compression by rank $c$ projection. 
We refer to \cite{VoDyNi92} and \cite{NiSp06} for notions in free probability. 

(4) In \cite{Hor25}, we considered averaged limit shapes also in the case of a pausing time 
obeying a (heavy-tailed) stable distribution and derived dynamical behavior of 
limit shapes different from \eqref{eq:1-36} (\cite[Proposition~4.4]{Hor25}). 
As we showed in \cite[Theorem~1.5]{Hor20}, however, if pausing time obeys a stable 
distributions, then initial AFP is essentially never propagated 
at positive macroscopic time under any space-time rescaling. 
Hence that model does not fall into the framework of concentrated limit shapes such as 
Theorem~\ref{th:1-3}. 
While \cite{Hor25} is a companion paper of the present one, we list some of their differences 
in Table~\ref{tab:1-2}.
\hfill $\square$

\medskip

The subsequent sections are devoted to proofs of Theorems~\ref{th:1-1}, \ref{th:1-2} 
and \ref{th:1-3}.

\section{Proof of Theorem~\ref{th:1-1}} 
We show Theorem~\ref{th:1-1} in three steps. 

\textit{Step~1}. In general, let $G$ be a finite group, $H$ its subgroup, and 
$\mathrm{E} : G \longrightarrow H$ the restriction map as \eqref{eq:1-6}. 
For $\alpha\in\widehat{H}$ and $b\in \mathbb{C}[G]$, we see from structure 
of group algebras  
\begin{equation}\label{eq:2-1}
\mathrm{tr} \pi^\alpha(\mathrm{E}b) = \frac{1}{[G:H] \dim\alpha} 
\sum_{\beta\in\widehat{G}} \dim\beta \, \mathrm{tr} \pi^\beta(b e^\alpha)
\end{equation}
where $\pi^\alpha$ [resp.$\;\pi^\beta$] is an irreducible representation of $H$ [resp.$\;G$], 
$e^\alpha$ is a minimal central projection in $\mathbb{C}[H]$, and $\mathrm{tr}$ denotes 
a (non-normalized) trace (see \cite[2.2]{Hor05}, \cite[Lemma~9.24]{HoOb07}). 
We apply \eqref{eq:2-1} to $G = \mathfrak{S}_{n+1}(T)$, $H=\mathfrak{S}_n(T)$, 
$\alpha = \lambda = (\lambda^\zeta)_{\zeta\in\widehat{T}}\in \mathbb{Y}_n(\widehat{T})$, 
and $b = J_{n+1}^{\ k} (A_\theta, (n\!+\!1)) = J_{n+1}^{\ k-1} J_{n+1}^\theta$. 

\textit{Step~2}. 
For $\mu = (\mu^\zeta)_{\zeta\in\widehat{T}}\in \mathbb{Y}_{n+1}(\widehat{T})$, 
let $(\pi^\mu, V^\mu)$ be a corresponding irreducible representation of $\mathfrak{S}_{n+1}(T)$. 
The action of $\mathfrak{S}_n(T)$ decomposes $V^\mu$ as 
\begin{equation}\label{eq:2-2}
V^\mu = \bigoplus_{\lambda\in\mathbb{Y}_n(\widehat{T}):\, \lambda\nearrow\mu} 
(W^\lambda)^{\oplus \dim\zeta_{\lambda\mu}}, \qquad 
W^\lambda \cong_{\mathfrak{S}_n(T)} V^\lambda.
\end{equation}
We note $\pi^\mu (J_{n+1}^\theta)$ acts as a scalar on $W^\lambda$ since 
$J_{n+1}^\theta$ commutes with $\mathfrak{S}_n(T)$ by \eqref{eq:1-13}. 
Here we find the scalar value through irreducible characters. 
Note $\pi^\mu(A_{\iota_{n+1, 2} (2)_\theta})$ [resp.
$\pi^\lambda (A_{\iota_{n, 2} (2)_\theta})$] acts on $V^\mu$ (hence on $W^\lambda$) 
[resp. on $W^\lambda$] of \eqref{eq:2-2} as a scalar 
\[ 
\frac{\chi^\mu_{\iota_{n+1, 2}(2)_\theta}}{\dim\mu} 
\lvert C_{\iota_{n+1, 2}(2)_\theta}\rvert \qquad \Bigl[ \text{resp.} \quad 
\frac{\chi^\lambda_{\iota_{n, 2}(2)_\theta}}{\dim\lambda} 
\lvert C_{\iota_{n, 2}(2)_\theta}\rvert \Bigr].
\] 
Character formula for a wreath product yields, in particular, 
\[ 
\frac{\chi^\lambda_{\iota_{n, k}(k)_\theta}}{\dim\lambda} = 
\frac{1}{n^{\downarrow k}} \sum_{\zeta\in\widehat{T}} 
\frac{\chi^\zeta_\theta}{(\dim\zeta)^k} \Sigma_k(\lambda^\zeta) 
\] 
(see \cite[(2.4)]{Hor25}) where 
\[ 
\Sigma_k(\lambda^\zeta) := \lvert\lambda^\zeta\rvert^{\downarrow k} 
\frac{\chi^{\lambda^\zeta}_{(k, 1^{\lvert\lambda^\zeta\rvert-k})}}{\dim\lambda^\zeta}
\] 
is an often used notation. 
Frobenius' formula tells us 
\begin{equation}\label{eq:2-6}
\Sigma_2(\nu) = 
\sum_{i=1}^d \bigl( b_i(b_i+1) - a_i(a_i+1)\bigr), \qquad \nu\in\mathbb{Y} 
\end{equation}
where $b_i := \nu_i-i$ and $a_i := \nu^\prime_i-i$ are the Frobenius coordinates of $\nu$ 
($d$ : number of diagonal boxes, $\nu^\prime$ : transposed $\nu$). 
Furthermore, 
\[ 
\lvert C_{\iota_{n, 2}(2)_\theta}\rvert = \frac{n(n-1)}{2} \lvert T\rvert \lvert C_\theta\rvert 
\] 
holds. 
Putting these together, we have for $\lambda\nearrow\mu$ 
\begin{align}
\frac{\chi^\mu_{\iota_{n+1, 2}(2)_\theta}}{\dim\mu} 
\lvert C_{\iota_{n+1, 2}(2)_\theta}\rvert - 
\frac{\chi^\lambda_{\iota_{n, 2}(2)_\theta}}{\dim\lambda} 
\lvert C_{\iota_{n, 2}(2)_\theta}\rvert 
&= \frac{\lvert T\rvert \lvert C_\theta\rvert}{2}
\frac{\chi^{\zeta_{\lambda\mu}}_\theta}{(\dim\zeta_{\lambda\mu})^2} 
\bigl( \Sigma_2(\mu^{\zeta_{\lambda\mu}})-
\Sigma_2(\lambda^{\zeta_{\lambda\mu}})\bigr) \notag \\ 
&= \frac{\lvert T\rvert \lvert C_\theta\rvert
\chi^{\zeta_{\lambda\mu}}_\theta}{(\dim\zeta_{\lambda\mu})^2} \, 
c(\mu^{\zeta_{\lambda\mu}}\setminus \lambda^{\zeta_{\lambda\mu}}) 
\label{eq:2-8}
\end{align}
where $c(\square)$ denotes the content of box $\square$. 
The second equality of \eqref{eq:2-8} follows from that, if $\nu, \Tilde{\nu}\in \mathbb{Y}$ 
satisfy $\nu\nearrow\Tilde{\nu}$, \eqref{eq:2-6} yields 
\[ 
\Sigma_2(\Tilde{\nu}) -\Sigma_2(\nu) = 2 c(\Tilde{\nu}\setminus\nu).
\] 

\textit{Step~3}. 
Since we showed that $\pi^\mu(J_{n+1}^\theta)$ acts on $W^\lambda$ as the scalar 
given by \eqref{eq:2-8}, so does $\pi^\mu(J_{n+1}^{\ k-1} J_{n+1}^\theta)$ as 
\begin{equation}\label{eq:2-10}
\lvert T\rvert^k \lvert C_\theta\rvert \frac{\chi^{\zeta_{\lambda\mu}}_\theta}
{(\dim\zeta_{\lambda\mu})^{k+1}} \, 
c(\mu^{\zeta_{\lambda\mu}}\setminus \lambda^{\zeta_{\lambda\mu}})^k.
\end{equation}
We come back to Step~1. 
Now \eqref{eq:2-1} yields 
\begin{equation}\label{eq:2-11}
\mathrm{tr} \pi^\lambda \bigl( \mathrm{E}\bigl[ J_{n+1}^{\ k} 
(A_\theta, (n\!+\!1))\bigr]\bigr) = \frac{1}{(n+1)\lvert T\rvert} 
\sum_{\mu\in \mathbb{Y}_{n+1}(\widehat{T})} \frac{\dim\mu}{\dim\lambda} \,
\mathrm{tr} \pi^\mu \bigl( J_{n+1}^{\ k} (A_\theta, (n\!+\!1))\bigr) 
\pi^\mu(e^\lambda). 
\end{equation}
Since $\pi^\mu(e^\lambda)$ is the projection onto the $\lambda$-isotypical component 
$(W^\lambda)^{\oplus \dim\zeta_{\lambda\mu}}$, we have from \eqref{eq:2-10} 
and \eqref{eq:2-11} 
\begin{equation}\label{eq:2-12}
\frac{\chi^\lambda\bigl( \mathrm{E}\bigl[ J_{n+1}^{\ k} 
(A_\theta, (n\!+\!1))\bigr]\bigr)}{\dim\lambda} 
= \frac{\lvert T\rvert^k \lvert C_\theta\rvert}{(n+1)\lvert T\rvert} 
\sum_{\zeta\in\widehat{T}} \sum_{\mu^\zeta:\, \lambda^\zeta\nearrow\mu^\zeta} 
\frac{\dim\mu}{\dim\lambda} \frac{\chi^\zeta_\theta}
{(\dim\zeta)^k} \, c(\mu^\zeta\setminus \lambda^\zeta)^k.
\end{equation}
Applying the dimension formula to \eqref{eq:2-12}, we continue as 
\[ 
\text{\eqref{eq:2-12}} = \sum_{\zeta\in\widehat{T}} 
\frac{\lvert T\rvert^{k-1}\lvert C_\theta\rvert}{(\dim\zeta)^{k-1}} \chi^\zeta_\theta  
\sum_{\mu^\zeta:\, \lambda^\zeta\nearrow\mu^\zeta} 
\frac{\dim\mu^\zeta}{(\lvert\lambda^\zeta\rvert +1) \dim\lambda^\zeta} 
c(\mu^\zeta\setminus\lambda^\zeta)^k 
= \sum_{\zeta\in\widehat{T}} 
\frac{\lvert T\rvert^{k-1}\lvert C_\theta\rvert}{(\dim\zeta)^{k-1}} \chi^\zeta_\theta  
M_k(\mathfrak{m}_{\lambda^\zeta}).
\] 
We have thus showed \eqref{eq:1-16}. 
Noting $(\chi^\zeta_\theta \sqrt{\lvert C_\theta\rvert / \lvert T\rvert})_{\zeta,\theta}$ 
is a unitary matrix, we see \eqref{eq:1-16} is equivalent to \eqref{eq:1-15}. 
This completes the proof of Theorem~\ref{th:1-1}.

\section{Proof of Theorem~\ref{th:1-2}} 
In this section, $T$ is assumed to be a finite abelian group. 

\subsection{Preparatory argument} 
In \eqref{eq:1-17}, we separate the extra terms and recognize their form as desired. 
Recalling 
\[ 
J_{n+1} = \sum_{i=1}^n \sum_{x\in T} (x^{-1}, (i)) (x, (\ast)) (i\ \ast), \qquad 
\ast := n+1, 
\] 
we have for $k\in\mathbb{N}$ and $y\in T$ 
\begin{align}
&J_{n+1}^{\ k} (y, (\ast)) 
= \sum_{i_1, \cdots, i_k =1}^n \sum_{x_1, \cdots, x_k\in T} 
(y, (\ast)) (x_k^{-1}, (i_k)) (x_k, (\ast)) (i_k\ \ast) \cdots 
(x_1^{-1}, (i_1)) (x_1, (\ast)) (i_1\ \ast) \notag \\ 
&= \sum_{i_1, \cdots, i_k =1}^n \sum_{x_1, \cdots, x_k\in T} 
\underbrace{(y, (\ast)) (x_k^{-1}, (i_k)) (x_k, (\ast)) (i_k\ \ast) \cdots 
(x_1^{-1}, (i_1)) (x_1, (\ast)) (i_1\ \ast) (i_1\ \ast) \cdots (i_k\ \ast)}_{\qquad (\star)\in T^{n+1}} 
\notag \\ 
&\qquad\qquad\qquad
\underbrace{(i_k\ \ast) \cdots (i_1\ \ast)}_{\qquad (\star\star)\in \mathfrak{S}_{n+1}}.
\label{eq:3-1-2}
\end{align}
Mapped by 
$\mathrm{E} : \mathbb{C}[\mathfrak{S}_{n+1}(T)] \longrightarrow 
\mathbb{C}[\mathfrak{S}_n(T)]$, 
each term in \eqref{eq:3-1-2} survives if and only if $(\star)\in T^n$ and 
$(\star\star)\in\mathfrak{S}_n$ hold. 
We know in \cite{Hor06} when there exists $i_1, \cdots, i_k$ such that 
$(\star\star)\in\mathfrak{S}_n$ holds. 

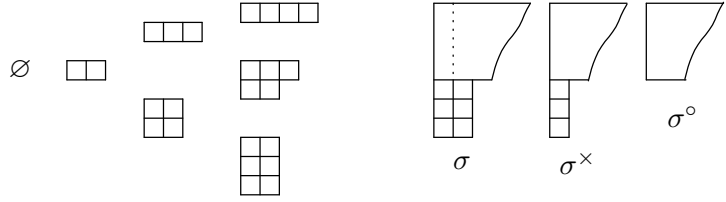
\begin{figure}[hbtp]
\centering
{\unitlength 0.1in%
\begin{picture}(37.0000,10.0500)(4.0000,-14.0000)%
%
\special{pn 8}%
\special{pa 700 700}%
\special{pa 900 700}%
\special{fp}%
\special{pa 900 700}%
\special{pa 900 800}%
\special{fp}%
\special{pa 900 800}%
\special{pa 700 800}%
\special{fp}%
\special{pa 700 800}%
\special{pa 700 700}%
\special{fp}%
\special{pa 800 700}%
\special{pa 800 800}%
\special{fp}%
\special{pa 1100 600}%
\special{pa 1100 500}%
\special{fp}%
\special{pa 1100 500}%
\special{pa 1400 500}%
\special{fp}%
\special{pa 1400 500}%
\special{pa 1400 600}%
\special{fp}%
\special{pa 1400 600}%
\special{pa 1100 600}%
\special{fp}%
\special{pa 1200 600}%
\special{pa 1200 500}%
\special{fp}%
\special{pa 1300 500}%
\special{pa 1300 600}%
\special{fp}%
\special{pa 1100 900}%
\special{pa 1300 900}%
\special{fp}%
\special{pa 1300 900}%
\special{pa 1300 1100}%
\special{fp}%
\special{pa 1300 1100}%
\special{pa 1100 1100}%
\special{fp}%
\special{pa 1100 1100}%
\special{pa 1100 900}%
\special{fp}%
\special{pa 1200 900}%
\special{pa 1200 1100}%
\special{fp}%
\special{pa 1100 1000}%
\special{pa 1300 1000}%
\special{fp}%
\special{pa 1600 400}%
\special{pa 2000 400}%
\special{fp}%
\special{pa 2000 400}%
\special{pa 2000 500}%
\special{fp}%
\special{pa 2000 500}%
\special{pa 1600 500}%
\special{fp}%
\special{pa 1600 500}%
\special{pa 1600 400}%
\special{fp}%
\special{pa 1700 400}%
\special{pa 1700 500}%
\special{fp}%
\special{pa 1800 500}%
\special{pa 1800 400}%
\special{fp}%
\special{pa 1900 400}%
\special{pa 1900 500}%
\special{fp}%
\special{pa 1600 700}%
\special{pa 1900 700}%
\special{fp}%
\special{pa 1900 700}%
\special{pa 1900 800}%
\special{fp}%
\special{pa 1900 800}%
\special{pa 1600 800}%
\special{fp}%
\special{pa 1600 700}%
\special{pa 1600 900}%
\special{fp}%
\special{pa 1600 900}%
\special{pa 1800 900}%
\special{fp}%
\special{pa 1800 900}%
\special{pa 1800 700}%
\special{fp}%
\special{pa 1700 700}%
\special{pa 1700 900}%
\special{fp}%
\special{pa 1600 1100}%
\special{pa 1800 1100}%
\special{fp}%
\special{pa 1800 1100}%
\special{pa 1800 1400}%
\special{fp}%
\special{pa 1800 1400}%
\special{pa 1600 1400}%
\special{fp}%
\special{pa 1600 1400}%
\special{pa 1600 1100}%
\special{fp}%
\special{pa 1700 1100}%
\special{pa 1700 1400}%
\special{fp}%
\special{pa 1600 1300}%
\special{pa 1800 1300}%
\special{fp}%
\special{pa 1800 1200}%
\special{pa 1600 1200}%
\special{fp}%
\put(4.0000,-8.0000){\makebox(0,0)[lb]{$\varnothing$}}%
%
\special{pn 8}%
\special{pa 2600 400}%
\special{pa 2600 1100}%
\special{fp}%
\special{pa 2600 1100}%
\special{pa 2800 1100}%
\special{fp}%
\special{pa 2800 1100}%
\special{pa 2800 800}%
\special{fp}%
\special{pa 2600 800}%
\special{pa 2900 800}%
\special{fp}%
\special{pa 2700 800}%
\special{pa 2700 1100}%
\special{fp}%
\special{pa 2600 1000}%
\special{pa 2800 1000}%
\special{fp}%
\special{pa 2800 900}%
\special{pa 2600 900}%
\special{fp}%
\special{pa 2600 400}%
\special{pa 3100 400}%
\special{fp}%
\special{pa 3200 400}%
\special{pa 3600 400}%
\special{fp}%
\special{pa 3200 400}%
\special{pa 3200 1100}%
\special{fp}%
\special{pa 3200 1100}%
\special{pa 3300 1100}%
\special{fp}%
\special{pa 3300 1100}%
\special{pa 3300 800}%
\special{fp}%
\special{pa 3200 800}%
\special{pa 3400 800}%
\special{fp}%
\special{pa 3200 900}%
\special{pa 3300 900}%
\special{fp}%
\special{pa 3300 1000}%
\special{pa 3200 1000}%
\special{fp}%
\special{pa 3700 400}%
\special{pa 4100 400}%
\special{fp}%
\special{pa 3700 400}%
\special{pa 3700 800}%
\special{fp}%
\special{pa 3900 800}%
\special{pa 3900 800}%
\special{fp}%
%
\special{pn 8}%
\special{pa 2700 400}%
\special{pa 2700 800}%
\special{dt 0.045}%
%
\special{pn 8}%
\special{pa 3700 800}%
\special{pa 3900 800}%
\special{fp}%
%
\special{pn 8}%
\special{pa 3900 795}%
\special{pa 3916 733}%
\special{pa 3926 703}%
\special{pa 3938 673}%
\special{pa 3951 644}%
\special{pa 3968 616}%
\special{pa 3986 590}%
\special{pa 4007 565}%
\special{pa 4026 539}%
\special{pa 4043 513}%
\special{pa 4057 484}%
\special{pa 4069 454}%
\special{pa 4083 425}%
\special{pa 4099 397}%
\special{pa 4100 395}%
\special{fp}%
%
\special{pn 8}%
\special{pa 3400 805}%
\special{pa 3416 743}%
\special{pa 3426 713}%
\special{pa 3438 683}%
\special{pa 3451 654}%
\special{pa 3468 626}%
\special{pa 3486 600}%
\special{pa 3507 575}%
\special{pa 3526 549}%
\special{pa 3543 523}%
\special{pa 3557 494}%
\special{pa 3569 464}%
\special{pa 3583 435}%
\special{pa 3599 407}%
\special{pa 3600 405}%
\special{fp}%
%
\special{pn 8}%
\special{pa 2900 800}%
\special{pa 2916 738}%
\special{pa 2926 708}%
\special{pa 2938 678}%
\special{pa 2951 649}%
\special{pa 2968 621}%
\special{pa 2986 595}%
\special{pa 3007 570}%
\special{pa 3026 544}%
\special{pa 3043 518}%
\special{pa 3057 489}%
\special{pa 3069 459}%
\special{pa 3083 430}%
\special{pa 3099 402}%
\special{pa 3100 400}%
\special{fp}%
\put(27.0000,-12.0000){\makebox(0,0)[lt]{$\sigma$}}%
\put(32.5000,-11.7000){\makebox(0,0)[lt]{$\sigma^\times$}}%
\put(38.1000,-9.4000){\makebox(0,0)[lt]{$\sigma^\circ$}}%
\end{picture}}%
\caption{Left: stratification of Young diagrams $\overline{\mathbb{Y}}$, \ 
Right: operations \eqref{eq:3-1-4} for $\sigma$}
\label{fig:3-1-1}
\end{figure}

Let us recall a certain random walk on Young diagrams used in \cite{Hor06}. 
We stratify $\overline{\mathbb{Y}} = \{ \tau\in\mathbb{Y} \,\vert\, m_1(\tau) =0\}$ 
according to the value of $\lvert\tau\rvert -l(\tau)$ (Figure~\ref{fig:3-1-1}, Left). 
Letting $n$ be large enough, consider the transition forming $(\star\star)$ 
\begin{equation}\label{eq:3-1-3}
e \to (i_1\ \ast) \to (i_2\ \ast)(i_1\ \ast) \to \cdots \to (i_k\ \ast)\cdots (i_1\ \ast), 
\quad i_1, i_2, \cdots, i_k\in\{1,2,\dots, n\}.
\end{equation}
The transition of cycle types of \eqref{eq:3-1-3} is a walk on $\overline{\mathbb{Y}}$ 
in Figure~\ref{fig:3-1-1} where each step goes up or down between neighboring strata. 
We define $\sigma^\times\in\mathbb{Y}$ and $\sigma^\circ\in\overline{\mathbb{Y}}$ for 
$\sigma = (2^{m_2(\sigma)}\cdots j^{m_j(\sigma)}\cdots)\in\overline{\mathbb{Y}}$ 
by (Figure~\ref{fig:3-1-1}, Right) 
\begin{equation}\label{eq:3-1-4}
\sigma^\times := (1^{m_2(\sigma)}\cdots j^{m_{j+1}(\sigma)}\cdots), \quad 
\sigma^\circ := (2^{m_3(\sigma)}\cdots j^{m_{j+1}(\sigma)}\cdots).
\end{equation}
For a given $\rho\in\overline{\mathbb{Y}}$, \cite[Lemma~1]{Hor06} tells us 
\begin{multline}\label{eq:3-1-5}
\exists i_1, \cdots, i_k\in \{1,\dots,n\} \text{ such that } (i_k\ \ast)\cdots (i_1\ \ast) 
\in \mathfrak{S}_n \text{ has type } \rho \\ 
\iff \exists \sigma\in\overline{\mathbb{Y}}_k \text{ such that } \rho = \sigma^\circ.
\end{multline}
In the walk of \eqref{eq:3-1-3} projected to their types, a step gaining a stratum 
(i.e. going to the right by one stratum in Figure~\ref{fig:3-1-1}, Left) is called an up step, while 
a step losing a stratum (i.e. going to the left by one) is called a down step. 
Given $\sigma^\circ\in \overline{\mathbb{Y}}$, let a $k$ step walk on 
$\overline{\mathbb{Y}}$ from $\varnothing$ to $\sigma^\circ$ have $u$ up steps and 
$d$ down steps. 
Since $u+d =k$ and $u-k = \lvert\sigma^\circ\rvert - l(\sigma^\circ)$ hold, we have 
\begin{equation}\label{eq:3-1-6}
d = \frac{k- (\lvert\sigma^\circ\rvert - l(\sigma^\circ))}{2} = l(\sigma), \qquad 
u= k-l(\sigma).
\end{equation}
If multiplication of $(i_j\ \ast)$ from left causes a down step, $i_j$ coincides with one 
of $i_1, \cdots, i_{j-1}$. 
On the other hand, if it is an up step, there are two possibilities. 
If $i_j$ does not coincide with any $i_1, \cdots, i_{j-1}$, let us call the step a strict-up. 
If $i_j$ coincides with one of $i_1, \cdots, i_{j-1}$, let us call the step a semi-up. 
When we count the number of walks according to power of $n$, a strict-up contributes 
by order $n$ and a semi-up contributes by order $1$. 
Hence, in counting the walks of \eqref{eq:3-1-3} which arrives at type $\sigma^\circ$ in 
\eqref{eq:3-1-5}, we have only to consider strict-up steps to get correct top terms in $n$ 
where the order is of $n^u = n^{k-l(\sigma)}$ by \eqref{eq:3-1-6}. 

A sequence $i_1, \cdots, i_k$ of \eqref{eq:3-1-3} corresponding to a walk with down 
and strict-up steps only and with a terminal of type $\sigma^\circ$ consisits of 
$l(\sigma)$ pairs $\{i_p, i_q\}$ and the other distinct singletons by \eqref{eq:3-1-6}. 
Since the terminal belongs to $\mathfrak{S}_n$, $(i_1, \cdots, i_k)$ produces a 
noncrossing partition of $\{1,\dots, k\}$ where indices $p<q$ of the $l(\sigma)$ pairs 
form the end points of its blocks. 
For example, letting $k=10$ and $i_5 =i_2$, $i_8=i_7$, $i_{10} =i_1$ with the distinct 
other $i_j$'s, we have an element of $\mathrm{NC}(10)$ in Figure~\ref{fig:1-2}. 
The $10$-step transition proceeds as 
\begin{align*}
e &\to (i_1\ \ast) \to (i_1\ i_2\ \ast) \to (i_1\ i_2\ i_3\ \ast) \to (i_1\ i_2\ i_3\ i_4\ \ast) 
\to (i_2\ i_3\ i_4)(i_1\ \ast) \\ 
&\to (i_2\ i_3\ i_4)(i_1\ i_6\ \ast) \to (i_2\ i_3\ i_4)(i_1\ i_6\ i_7\ \ast) \to 
(i_2\ i_3\ i_4)(i_7)(i_1\ i_6\ \ast) \\ 
&\to (i_2\ i_3\ i_4)(i_7)(i_1\ i_6\ i_9\ \ast) 
\to (i_2\ i_3\ i_4)(i_7)(i_1\ i_6\ i_9), 
\end{align*}
where the terminal type is $\sigma^\circ = (3^2)$ with $\sigma = (2^1 4^2)$, 
$\sigma^\times = (1^1 3^2)$. 

Applying $\mathrm{E}$ to \eqref{eq:3-1-2}, we pick up surviving terms. 
By virtue of the above arguments, we can divide them, for each 
$\sigma\in\overline{\mathbb{Y}}_k$, into the terms containing strict-ups only and the 
terms containing at least one semi-ups. 
The latter has strictly lower order than the former, and hence is of $O(n^{k-l(\sigma)-1})$. 
Recall the element $(\star)$ in \eqref{eq:3-1-2}. 
A surviving term has type $\sigma^\circ$ if it is projected to an element of $\mathfrak{S}_n$. 
However, since it carries elements of $T$ coming from $\mathrm{E}[(\star)]$, 
some $1$-box rows of $\sigma^\times$ may serve as supports in $\mathfrak{S}_n(T)$. 
This is the reason we used type $\sigma^\times$ (or with some $1$-box rows removed) 
for an element of $\mathfrak{S}_n(T)$ in the statement of Theorem~\ref{th:1-2}. 
Using $N_\sigma$ in \eqref{eq:1-20}, we thus obtain the following expression: 
\begin{equation}\label{eq:3-1-8}
\mathrm{E}[ J_{n+1}^{\ k} (y, (\ast)) ] = \sum_{\sigma\in\overline{\mathbb{Y}}_k} 
\Bigl\{ \sum_{(i_1,\cdots, i_k)\text{ forming } \mathrm{NC}(\sigma)} 
\sum_{x_1, \cdots, x_k\in T} \mathrm{E}[(\star)] h_{(i_1,\cdots, i_k)} 
+ \sum_{g\in N_\sigma} g \Bigr\}
\end{equation}
where $h_{(i_1,\cdots, i_k)}$ is an element of $\mathfrak{S}_n$ determined as the 
terminal of \eqref{eq:3-1-3}
\footnote{The type of $h_{(i_1,\cdots, i_k)}$ is $(\sigma^\circ, 1^{n-\lvert\sigma^\circ\rvert})
= (\sigma^\times, 1^{n-\lvert\sigma^\times\rvert})$. 
Viewed as the latter, $\sigma^\times$ is filled with the $k-l(\sigma)$ letters consisting of 
multiplicity-free $i_1, \cdots, i_k$.}. 

\subsection{Computation of the main terms for proof of Theorem~\ref{th:1-2}} 
Let us compute the main terms of \eqref{eq:3-1-8}. 
Fix $y\in T$, $\sigma\in \overline{\mathbb{Y}}_k$, a sequence of letters (each taken from 
$\{1,\dots, n\}$)  
$(i_1, \cdots, i_k)$ which produces an element of $\mathrm{NC}(\sigma)$, and 
$x_1, \cdots, x_k\in T$. 
We need to compute the $\mathrm{E}$-image of 
\begin{equation}\label{eq:3-2-1}
(y, (\ast)) (x_k^{-1}, (i_k)) (x_k, (\ast)) (i_k\ \ast) \cdots 
(x_1^{-1}, (i_1)) (x_1, (\ast)) (i_1\ \ast).
\end{equation}
The $\mathfrak{S}_n$-part of the $\mathrm{E}$-image of \eqref{eq:3-2-1} is known as 
\eqref{eq:3-1-8}, while it is not yet clear whether the $T^n$-part $\mathrm{E}[(\star)]$ 
survives or not. 
In computing $\mathrm{E}[\text{\eqref{eq:3-2-1}}]$, a fundamental piece is given by 
the case of one block: 
\begin{equation}\label{eq:3-2-2}
p<q, \quad i_q=i_p, \quad \text{and} \quad i_p, \cdots, i_{q-1} \text{ are distinct}.
\end{equation}
Since 
$(i_q\ \ast)\cdots (i_p\ \ast) = (i_p \ \cdots\ i_{q-1})$ holds under \eqref{eq:3-2-2}, 
elementary computation yields 
\begin{align}
&(x_q^{-1}, (i_q)) (x_q, (\ast)) (i_q\ \ast) \cdots (x_p^{-1}, (i_p)) (x_p, (\ast)) (i_p\ \ast) 
\notag \\ 
= &(x_qx_p^{-1}, (\ast))(x_q^{-1}x_{q-1}, (i_p)) 
\underbrace{(x_{p+1}^{-1}x_p, (i_{p+1})) \cdots (x_{q-1}^{-1}x_{q-2}, (i_{q-1}))}_{(\#)}
(i_p \ \cdots\ i_{q-1}). \label{eq:3-2-3}
\end{align}
In the $(\#)$ part of \eqref{eq:3-2-3}, it is clear how $(i_j)$ carries an element of $T$, 
and similarly for $(x_q^{-1}x_{q-1}, (i_p))$ since $i_q=i_p$ holds: 
\[
\begin{matrix}
(i_p) & (i_{p+1}) & \cdots & (i_r) & (i_{r+1}) & \cdots & (i_{q-1}) & (i_q=i_p) & \\ 
x_p^{-1} & x_{p+1}^{-1} & \cdots & x_r^{-1} & x_{r+1}^{-1} & \cdots & 
x_{q-1}^{-1} & x_q^{-1} &  \\ 
x_p & x_{p+1} & \cdots & x_r & x_{r+1} & \cdots & x_{q-1} & x_q & (\ast)
\end{matrix}
\] 
\begin{equation}\label{eq:3-2-4}
\Longrightarrow \quad 
\begin{matrix}
(i_{p+1}) & \cdots & (i_r) & (i_{r+1}) & \cdots & (i_{q-1}) & (i_p) & &  \\ 
x_{p+1}^{-1}x_p & \cdots & x_r^{-1}x_{r-1} & x_{r+1}^{-1}x_r & \cdots & 
x_{q-1}^{-1}x_{q-2} & x_q^{-1}x_{q-1} & & \\ 
 &  &   &  &  &  &  & x_q x_p^{-1} & (\ast).
\end{matrix} 
\end{equation}
The factor $ (x_qx_p^{-1}, (\ast))$ is sent to another computation for the block covering 
\eqref{eq:3-2-2}. 
We stratify the blocks of a noncrossing partition according to their depth. 

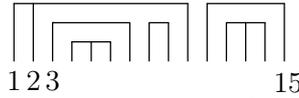
\begin{figure}[hbtp]
\centering
{\unitlength 0.1in%
\begin{picture}(14.3000,3.4500)(3.7000,-7.4500)%
%
\special{pn 8}%
\special{pa 400 400}%
\special{pa 1300 400}%
\special{fp}%
\special{pa 1300 400}%
\special{pa 1300 700}%
\special{fp}%
\special{pa 1200 700}%
\special{pa 1200 500}%
\special{fp}%
\special{pa 1200 500}%
\special{pa 1100 500}%
\special{fp}%
\special{pa 1100 500}%
\special{pa 1100 700}%
\special{fp}%
\special{pa 1000 700}%
\special{pa 1000 500}%
\special{fp}%
\special{pa 600 500}%
\special{pa 600 500}%
\special{fp}%
\special{pa 600 500}%
\special{pa 600 700}%
\special{fp}%
\special{pa 500 700}%
\special{pa 500 400}%
\special{fp}%
\special{pa 400 400}%
\special{pa 400 700}%
\special{fp}%
\special{pa 700 700}%
\special{pa 700 600}%
\special{fp}%
\special{pa 700 600}%
\special{pa 900 600}%
\special{fp}%
\special{pa 900 600}%
\special{pa 900 700}%
\special{fp}%
\special{pa 800 700}%
\special{pa 800 600}%
\special{fp}%
\special{pa 1000 500}%
\special{pa 600 500}%
\special{fp}%
\special{pa 1400 400}%
\special{pa 1400 700}%
\special{fp}%
\special{pa 1500 700}%
\special{pa 1500 500}%
\special{fp}%
\special{pa 1500 500}%
\special{pa 1700 500}%
\special{fp}%
\special{pa 1700 500}%
\special{pa 1700 700}%
\special{fp}%
\special{pa 1600 700}%
\special{pa 1600 500}%
\special{fp}%
\special{pa 1400 400}%
\special{pa 1800 400}%
\special{fp}%
\special{pa 1800 400}%
\special{pa 1800 700}%
\special{fp}%
\put(4.0000,-8.0000){\makebox(0,0){1}}%
\put(5.0000,-8.0000){\makebox(0,0){2}}%
\put(6.0000,-8.0000){\makebox(0,0){3}}%
\put(18.2000,-8.1000){\makebox(0,0){15}}%
\end{picture}}%
\caption{An element of $\mathrm{NC}(15)$ with 1st stratum: $\{4,5,6\}$; 
2nd stratum: $\{3,7\}$, $\{8,9\}$, $\{12,13,14\}$; 
3rd stratum: $\{1,2,10\}$, $\{11,15\}$}
\end{figure}

Firstly, we perform the computation of \eqref{eq:3-2-3} for the blocks in the first stratum 
like \eqref{eq:3-2-2}. 
As a result, we get 
\begin{equation}\label{eq:3-2-5}
(x_q^{-1}x_{q-1}, (i_p)) (x_{p+1}^{-1}x_p, (i_{p+1})) \cdots (x_{q-1}^{-1}x_{q-2}, (i_{q-1}))
(i_p \ \cdots\ i_{q-1}) \in \mathfrak{S}_n(T), 
\end{equation}
which is kept until the end without being affected by subsequent computations. 
If we look at a conjugacy class of $\mathfrak{S}_n(T)$, we can take 
\begin{equation}\label{eq:3-2-6}
x_q^{-1}x_{q-1}x_{p+1}^{-1}x_p \cdots x_{q-1}^{-1}x_{q-2} = x_q^{-1}x_p
\end{equation}
as an element of $T$ carried by the cycle $(i_p \ \cdots\ i_{q-1})$ in \eqref{eq:3-2-5}. 
After dealing with the blocks in the first stratum, we have some completed cycles in 
$\mathfrak{S}_n(T)$ like \eqref{eq:3-2-5} and elements carried by $(\ast)$ like 
$(x_qx_p^{-1}, (\ast))$. 
We have only to remember the product of \eqref{eq:3-2-6} for the elements of $T$ carried 
by a cycle of \eqref{eq:3-2-5} since we want to find which conjugacy class an element 
of $\mathfrak{S}_n(T)$ belongs to. 
Thus after the computation for the blocks 
$(i_p\ i_{p+1}\ \cdots\ i_{q-1}\ i_q=i_p)$, $(i_r\ i_{r+1}\ \cdots\ i_{s-1}\ i_s=i_r)$, 
$\cdots$ in the first stratum, there remain completed cycles 
\begin{equation}\label{eq:3-2-7}
\{ (x_q^{-1} x_p, (i_p)) (i_p\ \cdots \ i_{q-1}) \,\vert\, [p, q] 
\text{ is a block in the first stratum}\}
\end{equation}
and elements sent to next stage 
\begin{equation}\label{eq:3-2-8}
\{ (x_q x_p^{-1}, (\ast)) \,\vert\, [p, q] \text{ is a block in the first stratum}\}.
\end{equation}
The elements of \eqref{eq:3-2-8} are either left to the outer blocks in the second stratum 
or in the final stage if the first stratum is the outest. 

\begin{figure}[hbtp]
\centering
{\unitlength 0.1in%
\begin{picture}(47.1000,3.7000)(-1.5000,-7.7000)%
%
\special{pn 8}%
\special{pa 400 400}%
\special{pa 400 700}%
\special{fp}%
\special{pa 400 400}%
\special{pa 1600 400}%
\special{fp}%
\special{pa 1600 400}%
\special{pa 1600 700}%
\special{fp}%
\special{pa 1300 400}%
\special{pa 1300 700}%
\special{fp}%
\special{pa 1150 700}%
\special{pa 1150 550}%
\special{fp}%
\special{pa 1150 550}%
\special{pa 850 550}%
\special{fp}%
\special{pa 850 550}%
\special{pa 850 700}%
\special{fp}%
\special{pa 700 700}%
\special{pa 700 400}%
\special{fp}%
%
\special{pn 8}%
\special{pa 2610 410}%
\special{pa 2610 710}%
\special{fp}%
\special{pa 2910 710}%
\special{pa 2910 410}%
\special{fp}%
\special{pa 3060 560}%
\special{pa 3060 710}%
\special{fp}%
\special{pa 3060 560}%
\special{pa 3210 560}%
\special{fp}%
\special{pa 3210 560}%
\special{pa 3210 710}%
\special{fp}%
\special{pa 3360 710}%
\special{pa 3360 410}%
\special{fp}%
\special{pa 3660 410}%
\special{pa 3660 710}%
\special{fp}%
\special{pa 3810 710}%
\special{pa 3810 560}%
\special{fp}%
\special{pa 3810 560}%
\special{pa 4110 560}%
\special{fp}%
\special{pa 4110 560}%
\special{pa 4110 710}%
\special{fp}%
\special{pa 4260 710}%
\special{pa 4260 410}%
\special{fp}%
\special{pa 4560 410}%
\special{pa 4560 710}%
\special{fp}%
\special{pa 4560 410}%
\special{pa 2610 410}%
\special{fp}%
\put(4.3000,-7.4000){\makebox(0,0)[rt]{$p^\prime$}}%
\put(7.7000,-7.6000){\makebox(0,0)[lt]{$p$}}%
\put(11.5000,-7.5000){\makebox(0,0)[rt]{$q$}}%
\put(14.7000,-7.5000){\makebox(0,0)[lt]{$q^\prime$}}%
\put(26.4000,-7.5000){\makebox(0,0)[rt]{$p^\prime$}}%
\put(30.6000,-7.6000){\makebox(0,0)[rt]{$p$}}%
\put(32.4000,-7.6000){\makebox(0,0)[rt]{$q$}}%
\put(37.6000,-7.7000){\makebox(0,0)[lt]{$r$}}%
\put(41.3000,-7.6000){\makebox(0,0)[rt]{$s$}}%
\put(45.5000,-7.4000){\makebox(0,0)[lt]{$q^\prime$}}%
\end{picture}}%
\caption{inner block(s)}
\label{fig:3-2-2}
\end{figure}
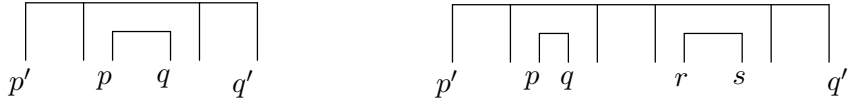

Secondly, we deal with blocks in the second stratum. 
If such a block has no inner blocks, it is treated in the same way as the first stratum. 
If $[p^\prime, q^\prime]$ holds an inner block $[p, q]$ (see Figure~\ref{fig:3-2-2}, Left), 
then as seen from \eqref{eq:3-2-4} we get a completed cycle 
\begin{equation}\label{eq:3-2-9}
(x_qx_p^{-1} x_{q^\prime}^{-1} x_{p^\prime}, (i_{p^\prime})) 
(i_{p^\prime}\ \cdots i_{p-1}\ i_{q+1}\ \cdots\ i_{q^\prime -1}) 
\end{equation}
(under product like \eqref{eq:3-2-6}) 
and $(x_{q^\prime} x_{p^\prime}^{-1}, (\ast))$ left. 
We note that $x_qx_p^{-1}$ in \eqref{eq:3-2-9}, which was sent from the inner block, 
is the inverse element of $x_q^{-1}x_p$, which was turned out in \eqref{eq:3-2-7}. 
If $[p^\prime, q^\prime]$ holds two inner blocks $[p, q]$ and $[r, s]$ 
(see Figure~\ref{fig:3-2-2}, Right), we get a completed cycle 
\begin{equation}\label{eq:3-2-10}
(x_q x_p^{-1} x_s x_r^{-1} x_{q^\prime}^{-1} x_{p^\prime}, (i_{p^\prime})) 
(i_{p^\prime}\ \cdots\ i_{p-1}\ i_{q+1}\ \cdots\ i_{r-1}\ i_{s+1}\ \cdots\ i_{q^\prime-1}) 
\end{equation}
and $(x_{q^\prime} x_{p^\prime}^{-1}, (\ast))$ left. 
In \eqref{eq:3-2-10} also, $x_q x_p^{-1} x_s x_r^{-1}$ sent from the inner blocks is the 
inverse of product of $x_q^{-1} x_p$ and $x_s^{-1} x_r$ turned out in dealing with 
the first stratum as \eqref{eq:3-2-7}. 
In general, as an element of $T$ carried by a completed cycle, we get the product of 
$x_qx_p^{-1}$'s over the inner blocks $[p, q]$ held by $[p^\prime, q^\prime]$, which are 
\begin{equation}\label{eq:3-2-11}
\text{the inverses of the elements of $T$ turned out in the first stratum treatment}, 
\end{equation}
multiplied by $x_{q^\prime}^{-1} x_{p^\prime}$. 
Also $(x_{q^\prime} x_{p^\prime}^{-1}, (\ast))$ is left, which is either sent to the 
third stratum or done if the second is the outest. 
Similar arguments proceed also for the other blocks $[r^\prime, s^\prime], \ \cdots$ in 
the second stratum. 

Our treatment comes to an end with the blocks $[t, u], \ [v, w], \ \cdots$ in the outest stratum. 
Then, carried by $(\ast)$, there remains 
\begin{equation}\label{eq:3-2-12}
(y x_u x_t^{-1} x_w x_v^{-1} \cdots, (\ast))
\end{equation}
left. 
Mapped by $\mathrm{E}$, \eqref{eq:3-2-12} survives if and only if 
\begin{equation}\label{eq:3-2-13}
x_u x_t^{-1} x_w x_v^{-1} \cdots = y^{-1}.
\end{equation}
Taking into account that the elements of $T$ turned out in dealing with each stratum 
fulfill the inverse relation of \eqref{eq:3-2-11}, we see \eqref{eq:3-2-13} holds 
if and only if 
\begin{equation}\label{eq:3-2-14}
(\text{product of the elements of $T$ carried by cycles in \eqref{eq:3-2-1} of 
type } \sigma^\times) =y.
\end{equation}

Let $x_1, \cdots, x_k$ run over $T^k$ in \eqref{eq:3-2-1} keeping 
$(i_1, \cdots, i_k)$ fixed. 
The indices of $x_j$'s appearing in the left hand side of the constraint of \eqref{eq:3-2-14} 
agree with the end points of the blocks of the noncrossing partition caused by 
$(i_1, \cdots, i_k)$. 
The other $k- 2l(\sigma)$ $x_j$'s are arbitrarily taken. 
Let $X_j, Y_k$, etc. denote those elements of $T$ which have expressions as 
$x_q x_p^{-1}$ ($p<q$). 
Using appropriate finite index sets $J, K, \cdots$, we obtain the following elements of $T$ 
carried by cycles of type $\sigma^\times$ after finishing treatments of all strata: 
\begin{align}
&\text{1st stratum} & &\{X_j^{-1}\}_{j\in J}, & &\ \notag \\ 
&\text{2nd stratum} & &\{Y_k^{-1} \prod_{j\in J_k}X_j\}_{k\in K}, & 
&J = \bigsqcup_k J_k, \notag \\ 
&\text{3rd stratum} & &\{Z_l^{-1} \prod_{k\in K_l}Y_k\}_{l\in L}, & 
&K = \bigsqcup_l K_l, \notag \\ 
&\qquad \vdots & &\ & &\ \notag \\ 
&\text{final stratum} & &\{V_n^{-1} \prod_{m\in M_n}U_m\}_{n\in N}, & 
&M = \bigsqcup_n M_n \label{eq:3-2-15}
\end{align}
(where some $J_k, K_l, \cdots$ may be empty) together with 
$\bigl(y \prod_{n\in N} V_n, (\ast)\bigr)$. 
Here all the members of 
\begin{equation}\label{eq:3-2-16}
\{X_j\}, \ \{Y_k\}, \ \{Z_l\}, \ \cdots, \ \{U_m\}, \ \{V_n\}
\end{equation}
have such forms as $x_q x_p^{-1}$ ($p<q$) with distinct $p, q$'s. 
If we naturally regard $x_j$'s as $T$-valued IID uniform random variables, 
so are \eqref{eq:3-2-16}. 
The number of random variables in \eqref{eq:3-2-16} agrees with the number of the 
blocks, that is $l(\sigma)$. 
We then see the random variables of \eqref{eq:3-2-15} are also IID and uniform on $T$. 

In order to determine a conjugacy class of $\mathfrak{S}_n(T)$, we put an element of $T$ 
on each of $l(\sigma)$ rows of $\sigma^\times$. 
Since the rows are all distinguished, there are $\lvert T\rvert^{l(\sigma)}$ configurations. 
In view of the argument of the previous paragraph, all configurations appear with the same 
probability. 
Among $x_1, \cdots, x_k \in T$, $2 l(\sigma)$ corresponding to ends of the blocks run 
in $\lvert T\rvert^{2l(\sigma)}$ ways. 
Hence they fall into each configuration by 
$\lvert T\rvert^{2l(\sigma)}/ \lvert T\rvert^{l(\sigma)}$. 
Moreover, $\lvert T\rvert^{k-2l(\sigma)}$ elements not corresponding to ends 
of the blocks run arbitrarily. 
Thus the number of the elements delivered to one configuration is 
\begin{equation}\label{eq:3-2-17}
\lvert T\rvert^{k-2l(\sigma)}\, \lvert T\rvert^{2l(\sigma)}/ \lvert T\rvert^{l(\sigma)}
= \lvert T\rvert^{k-l(\sigma)}.
\end{equation}
Among such configurations, one survives after mapped by $\mathrm{E}$ if and only if 
product of the elements of $T$ over the rows of $\sigma^\times$ coincides with $y$ 
as \eqref{eq:3-2-14}. 

Now we classify these configurations into the ones falling in the same conjugacy classes. 
For $x\in T$ and $k\in\mathbb{N}$, we recognize multiplicities coming from the 
permutations of $(x, k)$-cycles. 
For a row length $i$, choose $n_i^x$ rows out of the $m_i(\sigma^\times)$ rows 
carrying $x\in T^\circ$ ($:= T\setminus\{e_T\}$). 
Recall \eqref{eq:1-19}: 
\[ 
n_i^{e_T} := m_i(\sigma^\times) - \sum_{x\in T^\circ} n_i^x \qquad 
\text{so that} \qquad \sum_{x\in T} n_i^x = m_i(\sigma^\times).
\] 
The number of configurations falling into the same conjugacy class 
$C_{\iota_{n, \lvert\sigma^\times\rvert}((1^{n_1^x}2^{n_2^x}\cdots))_{x\in T}}$ 
is expressed as 
\begin{equation}\label{eq:3-2-19}
\binom{m_1(\sigma^\times)}{(n_1^x)_{x\in T}}
\binom{m_2(\sigma^\times)}{(n_2^x)_{x\in T}} \cdots 
\binom{m_i(\sigma^\times)}{(n_i^x)_{x\in T}} \cdots 
\end{equation}
by using multinomial coefficients. 
By \eqref{eq:3-2-17} and \eqref{eq:3-2-19}, the number of $x_1, \cdots, x_k \in T$ 
delivered to conjugacy class 
$C_{\iota_{n, \lvert\sigma^\times\rvert}((1^{n_1^x}2^{n_2^x}\cdots))_{x\in T}}$ is 
\begin{equation}\label{eq:3-2-20}
\lvert T\rvert^{k-l(\sigma)} \prod_{i\geqq 1} 
\binom{m_i(\sigma^\times)}{(n_i^x)_{x\in T}}.
\end{equation}
As the coefficient of the conjugacy indicator we get the number of \eqref{eq:3-2-20} 
divided by cardinality of the class. 

Let us take an element of $\mathrm{NC}(\sigma)$ for $\sigma\in\overline{\mathbb{Y}}_k$ 
and consider $(i_1, \cdots, i_k)$'s giving rise to it. 
Restricted to the ones corresponding to walks where only strict ups are allowed for up steps, 
such  $(i_1, \cdots, i_k)$'s count as $n^{k-l(\sigma)} (1+ O(1/n))$ ways. 
Therefore, multiplied by $\lvert \mathrm{NC}(\sigma)\rvert$, the terms of \eqref{eq:1-17} 
for $\sigma\in\overline{\mathbb{Y}}_k$ have been obtained. 
This completes the proof of Theorem~\ref{th:1-2}. 

\section{Proof of Theorem~\ref{th:1-3}} 
We begin with a remark. 
Let $(k)_\theta$ denote the type of $(\theta, k)$-cycles ($\theta\in [T]$, $k\in\mathbb{N}$): 
\[ 
\bigl( (k)_\theta\bigr)_{\theta^\prime} = \begin{cases} (k), & \theta^\prime =\theta \\ 
\varnothing, & \theta^\prime \neq\theta, \end{cases} \qquad (k)_\theta \in 
\mathbb{Y}_k([T]). 
\] 
If $\{ f^{(n)}\}_{n\in\mathbb{N}}$ satisfies AFP \eqref{eq:1-28} and 
\[ 
\lim_{n\to\infty} n^{\frac{k-1}{2}} f^{(n)} \bigl( \iota_{n,k} (k)_\theta\bigr), 
\qquad k\in\mathbb{N} 
\] 
exists, then we have for any $\rho\in \mathbb{Y}([T])$ 
\begin{equation}\label{eq:4-0-3}
f^{(n)} (\iota_{n, \lvert\rho\rvert} \rho ) = 
O\bigl( n^{-\frac{\lvert\rho\rvert - l(\rho)}{2}} \bigr) \quad (n\to\infty)
\end{equation}
by an inductive argument
\footnote{There is the same argument in \cite[Lemma~3.1]{Hor24}.}. 
Though \eqref{eq:4-0-3} holds for general $T$, we assume $T$ is abelian in this section below.

\subsection{Convergence of averaged moments of rescaled transition measures} 
In this subsection, we compute 
\[ 
\lim_{n\to\infty} n^{-k/2} \sum_{\lambda\in\mathbb{Y}_n(\widehat{T})} 
M_{tn}^{(n)}(\lambda) M_k( \mathfrak{m}_{\lambda^\zeta}) 
= \lim_{n\to\infty} \mathbb{E}_{M_{tn}^{(n)}} \bigl[ 
M_k( \mathfrak{m}_{(\lambda^\zeta)^{\sqrt{n}}}) \bigr] 
\] 
and capture \textit{averaged limit shapes} for multi-diagrams. 

\subsubsection{$t=0$ (initial ensemble)} 
Let $f_0^{(n)}$ act on \eqref{eq:1-17} and take limit of $n\to\infty$ after multiplying 
$n^{-k/2}$. 
By using \eqref{eq:4-0-3} and noting the definition of $n_i^{e_T}$ \eqref{eq:1-19}, we have 
\begin{multline}\label{eq:4-1-2}
f_0^{(n)} \bigl( \iota_{n, \lvert\sigma^\times\rvert}((1^{n_1^x}2^{n_2^x}\cdots))_{x\in T} 
\bigr) = \Bigl( \prod_{x\in T} f_0^{(n)} (\iota_{n,1} (1)_x )^{n_1^x}\Bigr) 
\Bigl( \prod_{x\in T} f_0^{(n)} (\iota_{n,2} (2)_x )^{n_2^x}\Bigr) \cdots \\ 
+ o\bigl( n^{- \frac{\lvert\sigma^\times\rvert -l(\sigma^\times)}{2}} \bigr) 
\qquad (n\to\infty).
\end{multline}
Here we have 
\begin{align}\label{eq:4-1-3}
&\frac{\lvert\sigma^\times\rvert -l(\sigma^\times)}{2} = 
\frac{\lvert\sigma\rvert - 2l(\sigma)}{2} = \frac{k}{2} - l(\sigma), \notag \\ 
&\text{hence } \qquad  
n^{-\frac{k}{2}} n^{k-l(\sigma)} 
o\bigl( n^{- \frac{\lvert\sigma^\times\rvert -l(\sigma^\times)}{2}} \bigr) = o(1) 
\quad (n\to\infty).
\end{align}
Combining \eqref{eq:4-1-2}, \eqref{eq:4-1-3} and \eqref{eq:1-31}, we get 
\begin{align}
&\lim_{n\to\infty} n^{-\frac{k}{2}} n^{k-l(\sigma)} 
f_0^{(n)} \bigl( \iota_{n, \lvert\sigma^\times\rvert}((1^{n_1^x}2^{n_2^x}\cdots))_{x\in T} 
\bigr) \notag \\ 
&= \lim_{n\to\infty} 
\prod_{x\in T} \Bigl( n^{\frac{1-1}{2}} f_0^{(n)} (\iota_{n,1} (1)_x )\Bigr)^{n_1^x}
\prod_{x\in T} \Bigl( n^{\frac{2-1}{2}} f_0^{(n)} (\iota_{n,2} (2)_x )\Bigr)^{n_2^x}
\prod_{x\in T} \Bigl( n^{\frac{3-1}{2}} f_0^{(n)} (\iota_{n,3} (3)_x )\Bigr)^{n_3^x}
\cdots \notag \\ 
&= \prod_{x\in T} \bigl( (\gamma_2^x)^{n_1^x} (\gamma_3^x)^{n_2^x} 
(\gamma_4^x)^{n_3^x}\cdots \bigr).  
 \label{eq:4-1-4}
\end{align}
Then \eqref{eq:1-17} yields for $y\in T$ 
\begin{align}
&\lim_{n\to\infty} n^{-\frac{k}{2}} f_0^{(n)} \bigl( \mathrm{E} [ J_{n+1}^{\ k} 
(y, (n+1))] \bigr) \notag \\ 
&= \sum_{\sigma\in\overline{\mathbb{Y}}_k} 
\sum_{(n_1^x, n_2^x, \cdots)_{x\in T^\circ} :\; (\star)_y} \!\!\!\!\!\!\!
\lvert \mathrm{NC}(\sigma)\rvert \lvert T\rvert^{k-l(\sigma)} 
\binom{m_1(\sigma^\times)}{(n_1^x)_{x\in T}} 
\binom{m_2(\sigma^\times)}{(n_2^x)_{x\in T}} \cdots 
\prod_{x\in T} \bigl(  (\gamma_2^x)^{n_1^x} (\gamma_3^x)^{n_2^x} \cdots\bigr). 
\label{eq:4-1-5}
\end{align}
Recall  \eqref{eq:1-15} as 
\begin{equation}\label{eq:4-1-6}
M_k(\mathfrak{m}_{\lambda^\zeta}) = \frac{1}{\lvert T\rvert^k} \sum_{y\in T} 
\frac{\chi^\lambda( \mathrm{E} [ J_{n+1}^{\;k} (y, (n+1))])}{\dim\lambda} 
\, \overline{\zeta(y)}, \qquad \zeta\in \widehat{T}.
\end{equation}
We obtain from \eqref{eq:4-1-5} and \eqref{eq:4-1-6} 
\begin{align}
&\lim_{n\to\infty} n^{-\frac{k}{2}} \sum_{\lambda\in\mathbb{Y}_n(\widehat{T})} 
M_0^{(n)}(\lambda) M_k(\mathfrak{m}_{\lambda^\zeta}) 
= \lim_{n\to\infty} n^{-\frac{k}{2}} \frac{1}{\lvert T\rvert^k} \sum_{y\in T} 
f_0^{(n)} \bigl( \mathrm{E} [ J_{n+1}^{\ k} (y, (n+1))] \bigr)  \overline{\zeta(y)} 
\notag \\ 
&= \sum_{\sigma\in\overline{\mathbb{Y}}_k} 
\frac{\lvert \mathrm{NC}(\sigma)\rvert}{\lvert T\rvert^{l(\sigma)}} \sum_{y\in T} 
\sum_{(n_1^x, n_2^x, \cdots)_{x\in T^\circ} :\; (\star)_y}
\binom{m_1(\sigma^\times)}{(n_1^x)_{x\in T}} 
\binom{m_2(\sigma^\times)}{(n_2^x)_{x\in T}} \cdots \notag \\ 
&\qquad\qquad\qquad\qquad\qquad \prod_{x\in T}\bigl( (\gamma_2^x)^{n_1^x}(\gamma_3^x)^{n_2^x} \cdots \bigr) 
\prod_{x\in T}\overline{\zeta(x)}^{n_1^x+n_2^x+\cdots} \notag \\
&= \sum_{\sigma\in\overline{\mathbb{Y}}_k} 
\frac{\lvert \mathrm{NC}(\sigma)\rvert}{\lvert T\rvert^{l(\sigma)}}
\sum_{(n_1^x, n_2^x, \cdots)_{x\in T} : \, \sum_{x\in T} n_1^x = m_1(\sigma^\times), 
\, \sum_{x\in T}n_2^x = m_2(\sigma^\times), \, \cdots} \notag \\ 
&\qquad\qquad\qquad\qquad\qquad\binom{m_1(\sigma^\times)}{(n_1^x)_{x\in T}} 
\prod_{x\in T} (\gamma_2^x \overline{\zeta(x)})^{n_1^x} \, 
\binom{m_2(\sigma^\times)}{(n_2^x)_{x\in T}}
\prod_{x\in T} (\gamma_3^x \overline{\zeta(x)})^{n_2^x} \, \cdots \notag \\ 
&= \sum_{\sigma\in\overline{\mathbb{Y}}_k} \lvert \mathrm{NC}(\sigma)\rvert 
\Bigl( \frac{1}{\lvert T\rvert} 
\sum_{x\in T} \gamma_2^x \overline{\zeta(x)}\Bigr)^{m_2(\sigma)} 
\Bigl( \frac{1}{\lvert T\rvert} 
\sum_{x\in T} \gamma_3^x \overline{\zeta(x)}\Bigr)^{m_3(\sigma)} \cdots.
\label{eq:4-1-7}
\end{align}
For any $\zeta\in\widehat{T}$ the leftmost side of \eqref{eq:4-1-7} 
\[ 
M_{0, k}^\zeta := \lim_{n\to\infty} \mathbb{E}_{M_0^{(n)}} \bigl[ 
M_k(\mathfrak{m}_{(\lambda^\zeta)^{\sqrt{n}}}) \bigr], \qquad 
k\in\mathbb{N}
\] 
gives a moment sequence $\{M_{0,k}^\zeta\}_{k=0,1,2,\dots}$ since it consists of limits 
of convex combinations of moments of a probability (i.e. it satisfies the condition of 
nonnegativity of the Hankel matrix). 
Hence there exists probability $\mathfrak{m}(0)^\zeta$ on $\mathbb{R}$ such that 
\begin{equation}\label{eq:4-1-9}
M_k( \mathfrak{m}(0)^\zeta) = M_{0,k}^\zeta, \qquad k\in\{0,1,2,\dots\}.
\end{equation}
Since \eqref{eq:1-32} and \eqref{eq:4-1-7} imply $(M_{0,k}^\zeta)$ of \eqref{eq:4-1-9} 
fulfills Carleman's condition, $\mathfrak{m}(0)^\zeta$ is uniquely determined. 
Also \eqref{eq:4-1-7} tells us that the free cumulants of $\mathfrak{m}(0)^\zeta$ are 
given by 
\[ 
R_1(\mathfrak{m}(0)^\zeta) =0, \qquad R_k(\mathfrak{m}(0)^\zeta) = 
\frac{1}{\lvert T\rvert} \sum_{x\in T} \gamma_k^x \overline{\zeta(x)} \quad 
(k\geqq 2). 
\] 
Note that $\sum_{x\in T} \gamma_k^x \overline{\zeta(x)}$ is real since 
$f_0^{(n)}(\iota_{n,k}(k)_{x^{-1}}) = \overline{f_0^{(n)}(\iota_{n,k}(k)_x)}$ holds. 
Let $\omega(0)^\zeta \in \mathscr{D}$ be the Markov transform of $\mathfrak{m}(0)^\zeta$. 
We have thus proved the following. 

\begin{proposition}\label{prop:4-1}
Under the initial condition \eqref{eq:1-31}--\eqref{eq:1-32} of Theorem~\ref{th:1-3}, 
there exists a unique \ 
$\omega(0) = (\omega(0)^\zeta)_{\zeta\in\widehat{T}}$, \ 
$\omega(0)^\zeta\in \mathscr{D}$, \  such that 
\[ 
M_k (\mathfrak{m}_{\omega(0)^\zeta}) = \lim_{n\to\infty} 
\mathbb{E}_{M_0^{(n)}} \bigl[ 
M_k (\mathfrak{m}_{(\lambda^\zeta)^{\sqrt{n}}}) \bigr], \qquad 
k\in\mathbb{N}, \ \zeta\in\widehat{T} 
\] 
holds and the free cumulants $R_k (\mathfrak{m}_{\omega(0)^\zeta})$ are given 
by \eqref{eq:1-35}.
\end{proposition}

\noindent\textbf{Remark} \ 
Proposition~\ref{prop:4-1} enables us to say we have captured the initial 
averaged limit shape $\omega(0)$.

\subsubsection{$t\geqq 0$ (any macroscopic time)} 
We show the following for $\{f_{tn}^{(n)}\}_{n\in\mathbb{N}}$: 
\begin{itemize}
\item AFP \eqref{eq:1-28} 
\item existence of limits for $(x, k)$-cycles analogous to \eqref{eq:1-31} and \eqref{eq:1-32}.
\end{itemize}
Then, similar computation to the case of $t=0$ will proceed. 

Properties of stochastic process $(X_s^{(n)})_{s\geqq 0}$ of \eqref{eq:1-24} produced 
by the Res-Ind chain are stated in our previous works. 
Recall stochastic matrix $P^{(n)}$ of \eqref{eq:1-22} and pausing time distribution $\psi$. 
Propositions~\ref{prop:4-2}, \ref{prop:4-3} and \ref{prop:4-4} below are valid for 
general (non-abelian) $T$. 

\begin{proposition}\label{prop:4-2}
For $\lambda, \mu \in \mathbb{Y}_n(\widehat{T})$, we have 
\begin{align}
&\mathbb{P} (X_s^{(n)} =\mu \,\vert\, X_0^{(n)} =\lambda) = \sum_{j=0}^\infty 
(P^{(n) j})_{\lambda\mu} \int_{[0, s]} \psi\bigl( (s-u, \infty)\bigr) 
\psi^{\ast j} (du),  \notag \\ 
&\text{therefore} \quad 
M_s^{(n)}(\lambda) = \sum_{j=0}^\infty (M_0^{(n)} P^{(n) j})_\lambda 
\int_{[0, s]} \psi\bigl( (s-u, \infty)\bigr) \psi^{\ast j} (du). \label{eq:4-1-13}
\end{align}
\end{proposition}
See \cite[(1.6)]{Hor25} and \cite[(1.6)]{Hor20}.

\begin{proposition}\label{prop:4-3}
Eigenvectors of $P^{(n)}$ are given by irreducible characters of $\mathfrak{S}_n(T)$: 
\begin{equation}\label{eq:4-1-14}
P^{(n)} \Bigl( \frac{\chi^\lambda_{\iota_{n,k}\rho}}{\dim\lambda} 
\Bigr)_{\lambda\in\mathbb{Y}_n(\widehat{T})} 
= \Bigl(1- \frac{k-m_1(\rho_{\{e_T\}})}{n}\Bigr) 
\Bigl( \frac{\chi^\lambda_{\iota_{n,k}\rho}}{\dim\lambda} 
\Bigr)_{\lambda\in\mathbb{Y}_n(\widehat{T})}, \qquad 
\rho\in\mathbb{Y}_k([T]).
\end{equation}
\end{proposition}
See \cite[(2.7)]{Hor25}, \cite[(3.3)]{Hor15} and \cite[(5.13)]{Hor16}. 

\medskip

For $\rho\in\mathbb{Y}_k([T])$ such that $m_1(\rho_{\{e_T\}}) =0$, 
\eqref{eq:4-1-13} and \eqref{eq:4-1-14} yield 
\begin{align*}
f_s^{(n)} (\iota_{n,k}\rho) &= \sum_{\lambda\in\mathbb{Y}_n(\widehat{T})} 
M_s^{(n)}(\lambda) \frac{\chi^\lambda_{\iota_{n,k}\rho}}{\dim\lambda} \\
&= \sum_{j=0}^\infty \Bigl( \int_{[0, s]} \psi\bigl( (s-u, \infty)\bigr) \psi^{\ast j} (du)\Bigr) 
M_0^{(n)} P^{(n) j} \Bigl( \frac{\chi^\lambda_{\iota_{n,k}\rho}}{\dim\lambda} 
\Bigr)_\lambda \\
&= a(k,n,s) f_0^{(n)}(\iota_{n,k}\rho) 
\end{align*}
where we set 
\[ 
a(k,n,s) = \sum_{j=0}^\infty \bigl( 1- \frac{k}{n}\bigr)^j 
\int_{[0, s]} \psi\bigl( (s-u, \infty)\bigr) \psi^{\ast j} (du).
\] 
We thus have for any $t\geqq 0$ and $k\in\mathbb{N}$ 
\begin{equation}\label{eq:4-1-17}
f_{tn}^{(n)} (\iota_{n,k}\rho) = a(k,n,tn) f_0^{(n)} (\iota_{n,k}\rho), \qquad 
\rho\in\mathbb{Y}_k([T]) : \ m_1(\rho_{\{e_T\}}) =0.
\end{equation}
Under the assumption of \eqref{eq:1-33}, 
\begin{equation}\label{eq:4-1-18}
\lim_{n\to\infty} a(k,n,tn) = e^{-kt/m}, \qquad t\geqq 0, \ k\in\mathbb{N} 
\end{equation}
holds. 
See \cite[Proposition 3.2 (1)]{Hor25} and \cite[Proposition 2.1]{Hor20}. 

\begin{proposition}\label{prop:4-4}
The sequence $\{f_{tn}^{(n)}\}_{n\in\mathbb{N}}$ fulfills AFP.
\end{proposition}
\textit{Proof} \ 
Let $g, h \in \mathfrak{S}_\infty(T)$ satisfy 
$\mathrm{supp}\,g \cap \mathrm{supp}\,h = \varnothing$, and set 
$\mathrm{type}(g) =\rho$, $\mathrm{type}(h) =\sigma$. 
Then, $\mathrm{type}(gh) = \rho\sqcup\sigma$. 
By \eqref{eq:4-1-17} and AFP for $\{f_0^{(n)}\}$, we have 
\begin{align*}
&f_{tn}^{(n)}(\iota_{n,\lvert\rho\rvert + \lvert\sigma\rvert}(\rho\sqcup\sigma)) 
- f_{tn}^{(n)}(\iota_{n,\lvert\rho\rvert}\rho) 
f_{tn}^{(n)}(\iota_{n, \lvert\sigma\rvert}\sigma) \\ 
&= \{ a(\lvert\rho\rvert + \lvert\sigma\rvert, n, tn) - 
a(\lvert\rho\rvert, n, tn) a(\lvert\sigma\rvert, n, tn) \} 
f_0^{(n)}(\iota_{n,\lvert\rho\rvert}\rho) f_0^{(n)}(\iota_{n, \lvert\sigma\rvert}\sigma) \\ 
&\qquad + a(\lvert\rho\rvert + \lvert\sigma\rvert, n, tn) 
o\bigl( n^{-\frac{1}{2}(\lvert\rho\rvert- l(\rho) + \lvert\sigma\rvert -l(\sigma))}\bigr) 
\qquad (n\to\infty), 
\end{align*}
which is 
$o\bigl( n^{-\frac{1}{2}(\lvert\rho\rvert- l(\rho) + \lvert\sigma\rvert -l(\sigma))}\bigr)$ 
by \eqref{eq:4-1-18} and \eqref{eq:4-0-3}.
\hfill $\blacksquare$

\medskip

For the value at $(x,k)$-cycle, \eqref{eq:4-1-17}, \eqref{eq:4-1-18} and \eqref{eq:1-31} 
yield 
\begin{equation}\label{eq:4-1-20}
\lim_{n\to\infty} n^{\frac{k-1}{2}} f_{tn}^{(n)} (\iota_{n,k} (k)_x) = 
e^{-kt/m} \gamma_{k+1}^x
\end{equation}
for $x\in T^\circ$, $k\in\mathbb{N}$ or $x=e_T$, $k\geqq 2$. 
If $x=e_T$ and $k=1$, the left hand side of \eqref{eq:4-1-20} is equal to 
$f_{tn}^{(n)}(e) = 1 = \gamma_2^{e_T}$. 
Under \eqref{eq:1-32} we have 
\begin{equation}\label{eq:4-1-21}
\sup_{k\geqq 2} \frac{1}{k}\, \lvert e^{-kt/m} \gamma_{k+1}^x\rvert^{1/k} 
< \infty.
\end{equation}
We recall the computation in 4.1.1 while we replace $\gamma_{k+1}^x$ by 
$e^{-kt/m}\gamma_{k+1}^x$ according to \eqref{eq:4-1-20} and, however, 
keep $\gamma_2^{e_T}$ for $x=e_T$ and $k=1$. 
The expression of \eqref{eq:4-1-4} is hence modified as 
\begin{multline*}
\lim_{n\to\infty} n^{-\frac{k}{2}} n^{k-l(\sigma)} 
f_{tn}^{(n)}\bigl( \iota_{n, \lvert\sigma^\times\rvert}((1^{n_1^x}2^{n_2^x}\cdots))_{x\in T} 
\bigr) \\ = \prod_{x\in T^\circ} (e^{-\frac{t}{m}} \gamma_2^x)^{n_1^x} 
\prod_{x\in T} (e^{-\frac{2t}{m}} \gamma_3^x)^{n_2^x} 
(e^{-\frac{3t}{m}} \gamma_4^x)^{n_3^x} \cdots. 
\end{multline*}
Then \eqref{eq:4-1-5} and \eqref{eq:4-1-7} turn into 
\begin{align}
&\lim_{n\to\infty} n^{-\frac{k}{2}} f_{tn}^{(n)} \bigl( \mathrm{E}[ 
J_{n+1}^{\ k} (y, (n+1))]\bigr) \notag\\ 
&= \sum_{\sigma\in\overline{\mathbb{Y}}_k} 
\sum_{(n_1^x, n_2^x, \cdots)_{x\in T^\circ} : \, (\star)_y} 
\lvert\mathrm{NC}(\sigma)\rvert\,\lvert T\rvert^{k-l(\sigma)}
\binom{m_1(\sigma^\times)}{(n_1^x)_{x\in T}} 
\binom{m_2(\sigma^\times)}{(n_2^x)_{x\in T}} \cdots \notag\\ 
&\qquad\qquad\qquad\prod_{x\in T^\circ} (e^{-\frac{t}{m}}\gamma_2^x )^{n_1^x} \, 
\prod_{x\in T} (e^{-\frac{2t}{m}}\gamma_3^x )^{n_2^x}
 (e^{-\frac{3t}{m}}\gamma_4^x )^{n_3^x}\cdots 
 \label{eq:4-1-23}
\end{align}
and 
\begin{align}
M_{t,k}^\zeta &= 
\lim_{n\to\infty} n^{-\frac{k}{2}} \sum_{\lambda\in\mathbb{Y}_n(\widehat{T})} 
M_{tn}^{(n)}(\lambda) M_k(\mathfrak{m}_{\lambda^\zeta}) \notag \\ 
&= \sum_{\sigma\in\overline{\mathbb{Y}}_k} 
\frac{\lvert\mathrm{NC}(\sigma)\rvert}{\lvert T\rvert^{l(\sigma)}} 
\sum_{(n_1^x, n_2^x, \cdots)_{x\in T} : \, \sum_{x\in T} n_1^x = m_1(\sigma^\times), 
\, \sum_{x\in T}n_2^x = m_2(\sigma^\times), \, \cdots} \notag \\ 
&\qquad\qquad\qquad\quad\binom{m_1(\sigma^\times)}{(n_1^x)_{x\in T}} 
\prod_{x\in T^\circ} (e^{-\frac{t}{m}}\gamma_2^x \overline{\zeta(x)})^{n_1^x} \, 
\binom{m_2(\sigma^\times)}{(n_2^x)_{x\in T}}
\prod_{x\in T} (e^{-\frac{2t}{m}}\gamma_3^x \overline{\zeta(x)})^{n_2^x} \, \cdots 
\notag \\ 
&= \sum_{\sigma\in\overline{\mathbb{Y}}_k} \lvert \mathrm{NC}(\sigma)\rvert 
\Bigl( \frac{1}{\lvert T\rvert} \bigl( 1+ 
\sum_{x\in T^\circ} e^{-\frac{t}{m}}\gamma_2^x \overline{\zeta(x)}\bigr) 
\Bigr)^{m_2(\sigma)} \Bigl( \frac{1}{\lvert T\rvert} 
\sum_{x\in T} e^{-\frac{2t}{m}}\gamma_3^x \overline{\zeta(x)}\Bigr)^{m_3(\sigma)} 
\cdots \label{eq:4-1-24}
\end{align}
respectively. 
Since \eqref{eq:4-1-21} implies Carleman's condition for $(M_{t,k}^\zeta)_k$ 
and make the moment problem determinate, probability $\mathfrak{m}(t)^\zeta$ on 
$\mathbb{R}$ is uniquely determined similarly to 4.1.1. 
Its free cumulants are read from \eqref{eq:4-1-24}: \ 
$R_1(\mathfrak{m}(t)^\zeta) =0$, 
\begin{align*}
R_2(\mathfrak{m}(t)^\zeta) &= \frac{1}{\lvert T\rvert} \bigl( 1+ 
\sum_{x\in T^\circ} e^{-\frac{t}{m}}\gamma_2^x \overline{\zeta(x)}\bigr) 
= \frac{1-e^{-\frac{t}{m}}}{\lvert T\rvert} + \frac{e^{-\frac{t}{m}}}{\lvert T\rvert} 
\sum_{x\in T} \gamma_2^x \overline{\zeta(x)},  \\ 
R_k(\mathfrak{m}(t)^\zeta) &= \frac{e^{-\frac{(k-1)t}{m}}}{\lvert T\rvert} 
\sum_{x\in T} \gamma_k^x \overline{\zeta(x)} \qquad (k\geqq 3). 
\end{align*}
Let $\omega(t)^\zeta$ be the Markov transform of $\mathfrak{m}(t)^\zeta$. 

\begin{proposition}\label{prop:4-5}
Under the assumptions of Theorem~\ref{th:1-3}, we have for any $t\geqq 0$ a unique 
limit shape $\omega(t) = (\omega(t)^\zeta)_{\zeta\in\widehat{T}}$, 
$\omega(t)^\zeta\in\mathscr{D}$, in an averaged sense such that 
\begin{equation}\label{eq:4-1-27}
M_k (\mathfrak{m}_{\omega(t)^\zeta}) = \lim_{n\to\infty} 
\mathbb{E}_{M_{tn}^{(n)}} \bigl[ M_k (\mathfrak{m}_{(\lambda^\zeta)^{\sqrt{n}}}) \bigr], 
\qquad k\in\mathbb{N}, \ \zeta\in\widehat{T}. 
\end{equation}
The free cumulants $R_k(\mathfrak{m}_{\omega(t)^\zeta})$ are given by \eqref{eq:1-36}.
\end{proposition}

\subsection{$L^2$-convergence of moments of rescaled transition measures}
We show the convergence of Proposition~\ref{prop:4-5} is taken to be in $L^2$, from which 
we obtain wLLN in Theorem~\ref{th:1-3}.

\subsubsection{Variance of moments of Jucys--Murphy elements}
We estimate the (co)variance of moments of JM elements w.r.t $\{f^{(n)}\}$: 
\begin{multline*}
f^{(n)} \bigl( \mathrm{E}[ J_{n+1}^{\ k} (y, (n+1))] 
\mathrm{E}[ J_{n+1}^{\ k} (z, (n+1))] \bigr) \\ - 
f^{(n)}\bigl( \mathrm{E}[ J_{n+1}^{\ k} (y, (n+1))] \bigr) 
f^{(n)}\bigl( \mathrm{E}[ J_{n+1}^{\ k} (z, (n+1))] \bigr). 
\end{multline*}
For $\rho, \sigma\in\mathbb{Y}([T])$ (in \eqref{eq:1-26.5}) 
and $r\in \{0,1,\dots, \lvert\rho\rvert\wedge\lvert\sigma\rvert\}$, set 
\[ 
S_{\rho, \sigma}^{(n)}(r) = \bigl\{ (g, h)\in C_{\iota_{n,\lvert\rho\rvert}\rho} \times 
C_{\iota_{n,\lvert\sigma\rvert}\sigma} \,\big\vert\, 
\lvert \mathrm{supp}\,g \cap \mathrm{supp}\,h \rvert =r \bigr\}.
\] 
Obviously we have 
\begin{equation}\label{eq:4-2-3}
\lvert S_{\rho, \sigma}^{(n)}(r) \rvert = O( n^{\lvert\rho\rvert + \lvert\sigma\rvert -r}) 
\quad (n\to\infty).
\end{equation}

\begin{lemma}\label{lem:4-2-1}
If $(g, h) \in S_{\rho, \sigma}^{(n)}(r)$ for $r\geqq 1$, then we have 
\begin{equation}\label{eq:4-2-4}
\lvert \mathrm{type} (gh)\rvert - l(\mathrm{type} (gh)) \geqq 
\lvert\rho\rvert -l(\rho) + \lvert\sigma\rvert -l(\sigma) -2r +2.
\end{equation}
\end{lemma}
\textit{Proof} \ 
For $g =dw$, $h=d^\prime w^\prime \in \mathfrak{S}_n(T)$ 
($d, d^\prime\in T^n$; $w, w^\prime \in \mathfrak{S}_n$), the $T^n$-part 
$dwd^\prime w^{-1}$ of $gh$ does not affect 
$\lvert \mathrm{type} (~)\rvert - l(\mathrm{type} (~))$. 
Hence \eqref{eq:4-2-4} follows from the case of $\mathfrak{S}_n$ proved in 
\cite[Lemma~3.5]{Hor05}
\footnote{In \cite{Hor05}, $l$ denotes length function $\lvert~\rvert - l(~)$. }.
\hfill $\blacksquare$

\begin{lemma}\label{lem:4-2-1}
If $\{f^{(n)}\}$ satisfies \eqref{eq:4-0-3}, we have 
\begin{equation}\label{eq:4-2-5}
f^{(n)}\Bigl( 
\frac{A_{\iota_{n,\lvert\rho\rvert}\rho}}{\lvert C_{\iota_{n,\lvert\rho\rvert}\rho}\rvert}\, 
\frac{A_{\iota_{n,\lvert\sigma\rvert}\sigma}}{\lvert C_{\iota_{n,\lvert\sigma\rvert}\sigma}\rvert}
\Bigr) - 
f^{(n)}\bigl( \iota_{n, \lvert\rho\rvert +\lvert\sigma\rvert}(\rho\sqcup\sigma)\bigr) 
= O\bigl( n^{-\frac{1}{2}(\lvert\rho\rvert-l(\rho)+\lvert\sigma\rvert- l(\sigma))-1}\bigr)
\end{equation}
as $n\to\infty$ for $\rho, \sigma\in\mathbb{Y}([T])$. 
\end{lemma}
\textit{Proof} \ 
According to 
\[ 
C_{\iota_{n,\lvert\rho\rvert}\rho} \times C_{\iota_{n,\lvert\sigma\rvert}\sigma} 
= \bigsqcup_{r=0}^{\lvert\rho\rvert\wedge\lvert\sigma\rvert} 
S_{\rho, \sigma}^{(n)}(r), \qquad 
A_{\iota_{n,\lvert\rho\rvert}\rho} A_{\iota_{n,\lvert\sigma\rvert}\sigma} 
= \sum_{r=0}^{\lvert\rho\rvert\wedge\lvert\sigma\rvert}  
\sum_{(g, h)\in S_{\rho, \sigma}^{(n)}(r)} gh, 
\] 
and \eqref{eq:4-0-3}, \eqref{eq:4-2-4}, we have 
\begin{align}
&f^{(n)} \bigl( 
A_{\iota_{n,\lvert\rho\rvert}\rho} A_{\iota_{n,\lvert\sigma\rvert}\sigma} \bigr) 
= \sum_{S_{\rho, \sigma}^{(n)}(0)} f^{(n)}(gh) + 
\sum_{r=1}^{\lvert\rho\rvert\wedge\lvert\sigma\rvert} 
\sum_{S_{\rho, \sigma}^{(n)}(r)} f^{(n)}(gh) \notag \\ 
&= \lvert S_{\rho, \sigma}^{(n)}(0)\rvert 
f^{(n)}\bigl( \iota_{n, \lvert\rho\rvert +\lvert\sigma\rvert}(\rho\sqcup\sigma)\bigr) 
+ \sum_{r=1}^{\lvert\rho\rvert\wedge\lvert\sigma\rvert} 
\sum_{S_{\rho, \sigma}^{(n)}(r)} 
O \bigl( n^{-\frac{1}{2}(\lvert\rho\rvert-l(\rho)+\lvert\sigma\rvert- l(\sigma)-2r+2)}\bigr). 
\label{eq:4-2-7}
\end{align}
By \eqref{eq:4-2-3}, \eqref{eq:4-2-7} yields \eqref{eq:4-2-5}.
\hfill $\blacksquare$

\medskip

Now we show for any $t\geqq 0$, $k\in \mathbb{N}$ and $y, z\in T$ 
\begin{multline}\label{eq:4-2-8}
f_{tn}^{(n)} \bigl( \mathrm{E}[ J_{n+1}^{\ k} (y, (n+1))] 
\mathrm{E}[ J_{n+1}^{\ k} (z, (n+1))] \bigr) \\ 
- f_{tn}^{(n)}\bigl( \mathrm{E}[ J_{n+1}^{\ k} (y, (n+1))] \bigr) \, 
f_{tn}^{(n)}\bigl( \mathrm{E}[ J_{n+1}^{\ k} (z, (n+1))] \bigr) 
= o(n^k) \qquad (n\to\infty).
\end{multline}
Since \eqref{eq:4-0-3} holds by AFP (Proposition~\ref{prop:4-4}) and \eqref{eq:4-1-20}, 
Lemma~\ref{lem:4-2-1} is applied. 
Theorem~\ref{th:1-2} gives us 
\begin{align}
&\mathrm{E}[ J_{n+1}^{\ k} (y, (n+1))] \mathrm{E}[ J_{n+1}^{\ k} (z, (n+1))] 
\notag \\ 
&= \sum_{\sigma, \tau\in \overline{\mathbb{Y}}_k} 
\sum_{\substack{(m_1^x, m_2^x, \cdots)_{x\in T^\circ}: \, (\star)_y \\ 
(n_1^x, n_2^x, \cdots)_{x\in T^\circ}: \, (\star)_z}} 
\lvert\mathrm{NC}(\sigma)\rvert \lvert\mathrm{NC}(\tau)\rvert 
n^{2k-l(\sigma)-l(\tau)} \bigl( 1+ O\bigl(\frac{1}{n}\bigr)\bigr) 
\lvert T\rvert^{2k-l(\sigma)-l(\tau)}
\notag \\ 
&\qquad\binom{m_1(\sigma^\times)}{(m_1^x)_{x\in T}} 
\binom{m_2(\sigma^\times)}{(m_2^x)_{x\in T}} \cdots 
\binom{n_1(\tau^\times)}{(n_1^x)_{x\in T}} 
\binom{n_2(\tau^\times)}{(n_2^x)_{x\in T}} \cdots 
\notag \\ 
&\qquad\frac{A_{\iota_{n, \lvert\sigma^\times\rvert}((1^{m_1^x}2^{m_2^x}\cdots))_{x\in T}}}
{\lvert C_{\iota_{n, \lvert\sigma^\times\rvert}((1^{m_1^x}2^{m_2^x}\cdots))_{x\in T}}\rvert}
\frac{A_{\iota_{n, \lvert\tau^\times\rvert}((1^{n_1^x}2^{n_2^x}\cdots))_{x\in T}}}
{\lvert C_{\iota_{n, \lvert\tau^\times\rvert}((1^{n_1^x}2^{n_2^x}\cdots))_{x\in T}}\rvert} 
\qquad + \ <\text{remainder terms}> 
\label{eq:4-2-9}
\end{align}
where $m_1^{e_T} = m_1(\sigma^\times) - \sum_{x\in T^\circ} m_1^x$, 
$n_1^{e_T} = m_1(\tau^\times) - \sum_{x\in T^\circ} n_1^x$. 
In \eqref{eq:4-2-9}, we have 
\begin{equation}\label{eq:4-2-10}
f_{tn}^{(n)} (<\text{remainder terms}>) = O(n^{k-1}) \quad (n\to\infty).
\end{equation}
In fact, \eqref{eq:4-0-3} and \eqref{eq:1-20} yield 
\begin{align*}
&\sum_{g\in N_\sigma} f_{tn}^{(n)}(g) = 
O\bigl( n^{-\frac{\lvert\sigma^\times\rvert -l(\sigma^\times)}{2}}\bigr) 
O(n^{k-l(\sigma)-1}) = O(n^{\frac{k}{2}-1}), \\ 
&n^{k-l(\sigma)} f_{tn}^{(n)} \bigl( 
\iota_{n, \lvert\sigma^\times\rvert}((1^{m_1^x}2^{m_2^x}\cdots))_{x\in T}\bigr) 
= n^{k-l(\sigma)} 
O\bigl( n^{-\frac{\lvert\sigma^\times\rvert -l(\sigma^\times)}{2}}\bigr) 
= O(n^{\frac{k}{2}}), 
\end{align*}
which imply \eqref{eq:4-2-10}. 
Noting 
\[ 
-\frac{1}{2}(\lvert\sigma^\times\rvert- l(\sigma^\times) +
\lvert\tau^\times\rvert -l(\tau^\times)) -1 
= -k +l(\sigma)+l(\tau) -1, 
\] 
we have from Lemma~\ref{lem:4-2-1} 
\begin{multline}\label{eq:4-2-14}
f_{tn}^{(n)} \Bigl( 
\frac{A_{\iota_{n, \lvert\sigma^\times\rvert}((1^{m_1^x}2^{m_2^x}\cdots))_{x\in T}}}
{\lvert 
C_{\iota_{n, \lvert\sigma^\times\rvert}((1^{m_1^x}2^{m_2^x}\cdots))_{x\in T}}\rvert} \, 
\frac{A_{\iota_{n, \lvert\tau^\times\rvert}((1^{n_1^x}2^{n_2^x}\cdots))_{x\in T}}}
{\lvert C_{\iota_{n, \lvert\tau^\times\rvert}((1^{n_1^x}2^{n_2^x}\cdots))_{x\in T}}
\rvert} \Bigr) \\ 
= f_{tn}^{(n)} \bigl( \iota_{n, \lvert\sigma^\times\rvert+\lvert\tau^\times\rvert}
((1^{m_1^x+n_1^x}2^{m_2^x+n_2^x}\cdots))_{x\in T}\bigr) 
+ O(n^{-k +l(\sigma)+l(\tau) -1}).
\end{multline}
Furthermore, AFP yields 
\begin{multline}\label{eq:4-2-15}
 f_{tn}^{(n)} \bigl( \iota_{n, \lvert\sigma^\times\rvert+\lvert\tau^\times\rvert}
((1^{m_1^x+n_1^x}2^{m_2^x+n_2^x}\cdots))_{x\in T}\bigr) \\
 =  f_{tn}^{(n)} \bigl( \iota_{n, \lvert\sigma^\times\rvert}
 ((1^{m_1^x}2^{m_2^x}\cdots))_{x\in T}\bigr) \, 
 f_{tn}^{(n)} \bigl( \iota_{n, \lvert\tau^\times\rvert}
 ((1^{n_1^x}2^{n_2^x}\cdots))_{x\in T}\bigr) 
 + o(n^{-k +l(\sigma)+l(\tau)}). 
\end{multline}
Combining \eqref{eq:4-2-14}, \eqref{eq:4-2-15} and \eqref{eq:4-2-10} with 
\eqref{eq:4-2-9}, we have 
\begin{multline*}
f_{tn}^{(n)} \bigl( \mathrm{E}[ J_{n+1}^{\ k} (y, (n+1))] 
\mathrm{E}[ J_{n+1}^{\ k} (z, (n+1))] \bigr) \\ 
= f_{tn}^{(n)}\bigl( \mathrm{E}[ J_{n+1}^{\ k} (y, (n+1))] \bigr) \, 
f_{tn}^{(n)}\bigl( \mathrm{E}[ J_{n+1}^{\ k} (z, (n+1))] \bigr) 
+ o(n^k) + O(n^{k-1}).
\end{multline*}
This completes the proof of \eqref{eq:4-2-8}. 

\subsubsection{Final step of Proof of Theorem~\ref{th:1-3}}
We show $L^2$-convergence of $M_k(\mathfrak{m}_{(\lambda^\zeta)^{\sqrt{n}}})$ 
for any $\zeta\in\widehat{T}$ and $k\in\mathbb{N}$. 
In order to see the convergence of \eqref{eq:4-1-23} is in $L^2$, let $C_k^y(t)$ 
denote the RHS of \eqref{eq:4-1-23}. 
Then, \eqref{eq:4-2-8} yields 
\begin{equation}\label{eq:4-2-17}
\lim_{n\to\infty} f_{tn}^{(n)} \bigl( ( n^{-\frac{k}{2}} 
\mathrm{E}[ J_{n+1}^{\ k} (y, (n+1))] - C_k^y(t)) 
( n^{-\frac{k}{2}} \mathrm{E}[ J_{n+1}^{\ k} (z, (n+1))] - C_k^z(t)) \bigr) =0.
\end{equation}
The RHS of \eqref{eq:4-1-24} is equal to 
\[ 
\frac{1}{\lvert T\rvert^k} \sum_{y\in T} C_k^y(t) \overline{\zeta(y)} 
=: K_k^\zeta(t).
\] 
Since a normalized irreducible character is multiplicative on the center, we have 
\begin{equation}\label{eq:4-2-19}
\mathbb{E}_{M_{tn}^{(n)}} \bigl[ (n^{-\frac{k}{2}} M_k( \mathfrak{m}_{\lambda^\zeta})
- K_k^\zeta(t))^2 \bigr] = 
f_{tn}^{(n)} \Bigl( \bigl\{ \sum_{y\in T} ( n^{-\frac{k}{2}} \mathrm{E} [ 
J_{n+1}^{\ k} (y, (n+1))] - C_k^y(t)) \overline{\zeta(y)} \bigr\}^2 \Bigr).
\end{equation}
Using the inequality for tracial state $\phi$ of $\ast$-algebra $\mathcal{A}$: 
\[ 
\phi\Bigl( \bigl( \sum_i a_i\bigr)^2 \Bigr) \leqq 
\Bigl( \sum_i \phi (a_i^\ast a_i)^{\frac{1}{2}} \Bigr)^2, \qquad a_i\in \mathcal{A}
\] 
and 
\[ 
\mathrm{E} [ J_{n+1}^{\ k} (y, (n+1))]^\ast = 
\mathrm{E} [ J_{n+1}^{\ k} (y^{-1}, (n+1))], \qquad 
\overline{C_k^y(t)} = C_k^{y^{-1}}(t), 
\] 
we continue \eqref{eq:4-2-19} as 
\[ 
\leqq \Bigl\{ \sum_{y\in T} f_{tn}^{(n)} \bigl( 
(n^{-\frac{k}{2}} \mathrm{E}[ J_{n+1}^{\ k} (y^{-1}, (n+1))] - C_k^{y^{-1}}(t)) 
(n^{-\frac{k}{2}} \mathrm{E}[ J_{n+1}^{\ k} (y, (n+1))] - C_k^y(t)) 
\bigr)^{\frac{1}{2}}\Bigr\}^2, 
\] 
which tends to $0$ as $n\to\infty$ by \eqref{eq:4-2-17}. 
We have thus obtained the $L^2$-convergence 
\begin{equation}\label{eq:4-2-23}
\lim_{n\to\infty} \mathbb{E}_{M_{tn}^{(n)}}\bigl[ \bigl( 
M_k(\mathfrak{m}_{(\lambda^\zeta)^{\sqrt{n}}}) - 
M_k(\mathfrak{m}_{\omega(t)^\zeta}) \bigr)^2 \bigr] =0
\end{equation}
where $\omega(t)^\zeta$ is the one in \eqref{eq:4-1-27}. 
In order to deduce \eqref{eq:1-34} from \eqref{eq:4-2-23}, we check the topology 
on $\mathscr{D}$, the continuous diagrams. 
The set $\mathscr{P}$ of probabilities on $\mathbb{R}$ with weak convergence topology 
is homeomorphic to $\mathscr{D}$ with uniform topology. 
Set 
\[ 
\mathscr{P}_0 = \{ \mathfrak{m}\in \mathscr{P} \,\vert\, \mathfrak{m} 
\text{ has all moments and its moment problem is determinate} \}.
\] 
Let $\mathscr{D}_0$ ($\subset \mathscr{D}$) be the image of $\mathscr{P}_0$. 
We equip $\mathscr{D}_0$ with the moment topology, which is metrizable topology 
defined by the countable family of semi-distances 
\[ 
\lvert M_k(\mathfrak{m}_{\varphi_1}) - M_k(\mathfrak{m}_{\varphi_2})\rvert, 
\qquad \varphi_1, \varphi_2 \in\mathscr{D}_0.
\] 
The moment topology is stronger than (or equal to) the uniform topology on $\mathscr{D}_0$. 
See \cite[Proposition~3.4]{Hor16} for reference. 
Namely, for $\varphi\in\mathscr{D}_0$, we have 
$\forall\varepsilon >0$, $\exists\delta >0$, $k_1, \cdots, k_p\in \mathbb{N}$ such that, 
$\forall\phi\in\mathscr{D}_0$, 
\[ 
\max_{i\in\{1,\dots, p\}} 
\lvert M_{k_i}(\mathfrak{m}_\phi) - M_{k_i}(\mathfrak{m}_\varphi) \rvert 
\leqq \delta \ \Longrightarrow \ \lVert \phi- \varphi \rVert_{\sup} \leqq \varepsilon.
\] 
In particular, taking $\phi = \nu^{\sqrt{n}}$, $\nu\in\mathbb{Y}\subset\mathscr{D}_0$, 
we have 
\begin{equation}\label{eq:4-2-27}
\bigl\{ \nu\in\mathbb{Y} \,\big\vert\, \lVert \nu^{\sqrt{n}} -\varphi\rVert_{\sup} 
> \varepsilon \bigr\} \subset 
\bigl\{ \nu\in\mathbb{Y} \,\big\vert\, \max_{i\in\{1,\dots, p\}} 
\lvert M_{k_i}(\mathfrak{m}_{\nu^{\sqrt{n}}}) - M_{k_i}(\mathfrak{m}_\varphi) \rvert 
> \delta \bigr\}.
\end{equation}
Let $\varphi$ be $\omega(t)^\zeta$ in \eqref{eq:4-2-23}. 
Then, \eqref{eq:4-2-27} yields 
\begin{multline}\label{eq:4-2-28}
\bigl\{ \lambda = (\lambda^\zeta)_{\zeta\in\widehat{T}} \in\mathbb{Y}_n(\widehat{T}) 
\,\big\vert\, \max_{\zeta\in\widehat{T}} \lVert (\lambda^\zeta)^{\sqrt{n}} 
- \omega(t)^\zeta \rVert_{\sup} > \varepsilon \bigr\} \\ 
\subset \bigcup_{\zeta\in\widehat{T}} \bigcup_{i=1}^p 
\bigl\{ \lambda \in\mathbb{Y}_n(\widehat{T}) 
\,\big\vert\, \lvert M_{k_i}(\mathfrak{m}_{(\lambda^\zeta)^{\sqrt{n}}}) - 
M_{k_i}(\mathfrak{m}_{\omega(t)^\zeta}) \rvert > \delta \bigr\}.
\end{multline}
Taking $M_{tn}^{(n)}$-values at \eqref{eq:4-2-28} and using \eqref{eq:4-2-23}, 
we obtain \eqref{eq:1-34}.

\subsection{Typical examples} 
As is stated before the definition of AFP, an extremal case of AFP is produced by character 
$f$ of $\mathfrak{S}_\infty (T)$, i.e. an extremal normalized positive-definite central 
function, through $f^{(n)} = f\vert_{\mathfrak{S}_n(T)}$. 
The characters of $\mathfrak{S}_\infty (T)$ for finite $T$ were classified in \cite{Boy05} 
and \cite{HiHi05} while we refer to \cite{HiHi07} and \cite{HiHiHo09} for general compact $T$. 
Since the limit shape of Young diagrams was first considered for the Plancherel measure, 
which corresponds to $f =\delta_e$, it is natural to discuss ensembles coming from 
characters of $\mathfrak{S}_\infty (T)$. 
The set of characters of $\mathfrak{S}_\infty (T)$ is parametrized by an extension of 
the Thoma simplex: 
\begin{align*}
\Delta = \Bigl\{ &(\alpha, \beta, c) = 
\bigl( (\alpha^\zeta_i)_{\zeta\in\widehat{T}, i\in\mathbb{N}}, 
(\beta^\zeta_i)_{\zeta\in\widehat{T}, i\in\mathbb{N}}, 
(c^\zeta)_{\zeta\in\widehat{T}} \bigr) \,\Big\vert\, \\ 
&\alpha^\zeta_1\geqq\alpha^\zeta_2\geqq\cdots\geqq 0, \ 
\beta^\zeta_1\geqq\beta^\zeta_2\geqq\cdots\geqq 0, \ c^\zeta \geqq 0, \ 
\sum_{i=1}^\infty (\alpha^\zeta_i +\beta^\zeta_i) \leqq c^\zeta, \ 
\sum_{\zeta\in\widehat{T}} c^\zeta =1 \Bigr\}. 
\end{align*}
To $\delta_e$ corresponds the parameters 
\[ 
\alpha^\zeta_i = \beta^\zeta_i =0, \qquad c^\zeta = (\dim\zeta)^2 /\lvert T\rvert.
\] 
The $\alpha$ and $\beta$ parameters are expressed as limits of ratios of row and column 
lengths divided by sizes of Young diagrams. 
Since we construct instead ensembles in which limit shapes survive, we are led to consider a 
regime of parameters satisfying 
\[ 
\alpha^\zeta_i = \alpha^\zeta_i(n) = O\bigl( \frac{1}{\sqrt{n}}\bigr), \ 
\beta^\zeta_i = \beta^\zeta_i(n) = O\bigl( \frac{1}{\sqrt{n}}\bigr), \ 
c^\zeta = c^\zeta(n) = O(1) \quad (n\to\infty)
\] 
and set 
\[ 
f^{(n)} = f_{\alpha(n), \beta(n), c(n)}\vert_{\mathfrak{S}_n(T)}, \qquad 
n\in\mathbb{N}.
\] 
Concrete computations for these ensembles are given in \cite[Section~5]{Hor25}. 
We omit details here to avoid duplication. 
In \cite{Hor25} we treated general averaged limit shapes. 
Part of the examples there serve as examples for concentrated limit shapes in this paper. 
Those include the case of abelian $T$ and the items in \cite[Table~1]{Hor25} for which 
pausing time is referred to as \lq\lq exponential-like\rq\rq.


\end{document}